\documentclass[10pt,hidelinks]{article}

\usepackage{cancel}

\usepackage[normalem]{ulem}

\usepackage{authblk}
\author[1]{Katherine Orme\~no Bast\'ias\thanks{supported in part by ANID beca
    doctorado nacional 21202166 and EPSRC grant EP/W007509/1}}
\author[2]{Paul Martin\thanks{supported in part by EPSRC grant EP/W007509/1}}
\author[3]{Steen Ryom-Hansen\thanks{supported in part by FONDECYT grant 1221112
and EPSRC grant EP/W007509/1}}

\affil[1]{Instituto de Matem\'aticas, Universidad de Talca, Chile}
 \affil[2]{School of Mathematics, Leeds University, UK}
  \affil[3]{Instituto de Matem\'aticas, Universidad de Talca, Chile }
\usepackage{cancel}
\usepackage{pstricks-add}
\usepackage{yfonts}
\usepackage{pstricks-add}

\usepackage{pst-func}

\usepackage{bbold}
\usepackage[normalem]{ulem}
\usepackage{etex}
\usepackage[left=2cm,top=2cm,right=2cm,bottom=2cm]{geometry}
\usepackage{amsmath, amsthm, amssymb}
\allowdisplaybreaks
\usepackage{pst-all}
\usepackage[textwidth=0.5in]{todonotes}
\usepackage[rightcaption]{sidecap}

\usepackage{relsize}
\usepackage{float}
\usepackage[caption = false]{subfig}

\usepackage[hypertexnames=false]{hyperref}

\usepackage[scaled=1]{helvet}
\usepackage{titlesec}
\usepackage{multido}
\usepackage{graphicx}
\usepackage{multirow}
\usepackage{blindtext}
\usepackage{color}
\usepackage{lipsum}
\usepackage{mathtools}
\usepackage[vcentermath]{youngtab}
\usepackage{young}
\usepackage[all,ps,dvips,graph]{xy}
\usepackage[misc,geometry]{ifsym}

\usepackage{titlesec}
\usepackage{lineno,hyperref}

\let\OLDthebibliography\thebibliography
\renewcommand\thebibliography[1]{
  \OLDthebibliography{#1}
  \setlength{\parskip}{0pt}
  \setlength{\itemsep}{0pt plus 0.3ex}
}

\titleformat{\section} {\normalfont\scshape \large \centering}{ \thesection}{1em}{}
\definecolor{morado}{rgb}{0.5,0,0.5}

\newcommand{\coeff}{{\rm coeff} }
\newcommand{\rad}{{\rm rad} }

\newcommand{\PP}{\Lambda }

\newcommand{\QQ}{\mathbb Q}

\newcommand{\ord}{{\rm{ord}}}
\newcommand{\Ord}{{{ord}}}
\newcommand{\Sign}{{\rm{sign}}}

\newcommand{\II}{\mathcal{I}}
\newcommand{\shape}{\text{shape}}
\newcommand{\norm}{\text{N}}
\newcommand{\normC}{{ \bf N}}

\newcommand{\ParAlg}{{\mathcal P }_{k}}
\newcommand{\SpheAlg}{{\mathcal{SP} }_{k}}
\newcommand{\SpheAlgThree}{{\mathcal{SP} }_{3}}
\newcommand{\AntiSpheAlg}{{\mathcal{ASP} }_{k}}

\newcommand{\N}{ { \mathbb N}}

\newcommand{\Z}{\mathbb{Z}}

\newcommand{\Par}{{\rm Par}   }
\newcommand{\Comp}{{\rm Comp}   }

\newcommand{\ParSph}{  \Par_{sph}^{k, n}}
\newcommand{\ParPar}{  \Par^{k, n}_{par}}
\newcommand{\ParSphThreeSix}{  \Par_{sph}^{3, 6}}

\newcommand{\CC}{ \mathbb C }

\newcommand{\End}{{\rm End}}
\newcommand{\Hom}{{\rm Hom}}
\newcommand{\Ext}{{\rm Ext}}
\newcommand{\spa}{{\rm span}}

\newcommand{\s}{\mathfrak{s}}

\newcommand{\cc}{\mathfrak{c}}
\newcommand{\dd}{\mathfrak{d}}

\newcommand{\V}{\mathfrak{v}}
\newcommand{\T}{  \mathfrak{t}}
\newcommand{\U}{  \mathfrak{u}}

\newcommand{\ch}{{\rm{ch}}}

\newcommand{\Si}{\mathfrak{S}}

\newcommand{\std}{{\rm Std}}
\newcommand{\sstd}{{\rm SStd}}

\newcommand{\tab}{{\rm Tab}}

\newcommand{\rank }{{ \rm rk }}

\newcommand{\Std}{{\rm Std}}

\newcommand{\BiPar}{ \operatorname{BiPar}}
\newcommand{\SetPar}{ \operatorname{SetPar}}

\modulolinenumbers[5]

\newcommand{\ppm}[1]{\textcolor{black}{#1}}
\newcommand{\ppmG}[1]{\textcolor{black}{#1}}

\newtheorem{theorem}{Theorem}
\newtheorem{lemma}[theorem]{Lemma}

\newtheorem{definition}[theorem]{Definition}
\newtheorem{corollary}[theorem]{Corollary}
\newtheorem{example}{Example}

\newtheorem{remark}{Remark}

\newenvironment{dem}{\noindent \textit{Proof:} }{\quad \hfill $\square$}

\numberwithin{equation}{section}

\newgray{plomoclaro}{.90}
\newgray{plomooscuro}{.70}
\newgray{blanco}{1}

\begin{document}
\Yvcentermath1
\sidecaptionvpos{figure}{lc}

\title{On the spherical partition algebra }

\date{\vspace{-5ex}}
\maketitle
\begin{abstract}
  {\color{black}{
      For $ k \in \mathbb{N} $ we introduce an idempotent subalgebra, the spherical partition algebra ${\mathcal{SP} }_{k}$, 
      of the partition algebra $ {\mathcal{P} }_{k} $, 
      that we define using an embedding associated with
      the trivial representation of the symmetric group $ \Si_k $.
We determine a basis for ${\mathcal{SP}} _{k}$, and this
provides a combinatorial interpretation of the dimension
\ppm{of} 
${\mathcal{SP}} _{k}$ involving
bipartite partitions of $k$.}}
For $ t \in \mathbb{C} $ we consider the specialized algebra $\mathcal{SP}_{k}(t)$. 
For $ t = n \in \mathbb{N}$, we describe the structure of $\mathcal{SP}_{k}(n)$
by giving the permutation module decomposition of the $k^{\color{black} \rm th}\, $symmetric power of
the defining module for the symmetric group algebra $ \mathbb{C} \mathfrak{S}_n $.
In general, we show that $\mathcal{SP}_{k}(t)$
  is quasi-hereditary over $ \mathbb{C}$ for all $ t \in \mathcal{C}$, except $ t=0$. 
  We determine the decomposition numbers for $\mathcal{SP}_{k}(t)$
  for every specialization $ t \in \mathbb{C} $
except $ t= 0 $, 
(which includes semisimple and non-semisimple cases). In particular we determine the structure of all indecomposable projective
modules{\color{black}{,}} and the indecomposable tilting modules.
\end{abstract}

\section{Introduction}
The {\it partition algebra} $ \ParAlg  $ arose around thirty years ago 
in the second named author's work on the Potts model in statistical mechanics, 
see \cite{PMartin}. 
Since then it has been understood that $ \ParAlg $ is in fact connected 
with many other areas of mathematics
and physics, including Deligne's category $ \underline{\rm Rep}(S_t) $, 
the Kronecker problem in the representation theory of 
the symmetric group, Schur algebras, and symmetric function theory, see for example
\cite{BOdV}, \cite{CO}, \cite{MartinWoodcock}, \cite{OrellanaZabrocki}.

\medskip
By definition, $ \ParAlg $ is a $\CC[x]$-algebra with basis
indexed by the 
set partitions on $  \{ 1,2 \ldots, k \} \cup 
\{ 1^{\prime},2^{\prime} \ldots, k^{\prime} \}$.
{\color{black}{There are many important 
subalgebras of the partition algebra, including
the half-integer partition algebra}}, the quasi-partition algebra,
the Temperley-Lieb algebra, the Motzkin algebra, 
the Brauer algebra, the quasi-Brauer algebra, the Rook algebra, the group algebra of
the symmetric group $\Si_k $, and so on,
see for example \cite{CST}, \cite{OrellanaDaugherty}, \cite{Scrimshaw} and references therein,
and it is also closely related to the $bt$-algebra of knot theory, see 
\cite{AicardiJuyumaya}, \cite{ArcisEspinoza}, \cite{Banjo}, \cite{ERH}, \cite{RH1}, \cite{RH2}. 
  
\medskip
In the present paper we introduce and study yet another subalgebra
of $ \ParAlg $, that we call the {\it spherical partition algebra} $ \SpheAlg  $.
By definition, $ \SpheAlg  $ is the $ \CC[x] $-algebra given by idempotent truncation of $ \ParAlg$, as follows
\begin{equation}
 \SpheAlg = e_k \ParAlg e_k 
\end{equation}  
where $  e_k = \iota_k\left ( \frac{1}{k! } \sum_{ \sigma \in \Si_k } \sigma\right)  $
and $ \iota_{\color{black}{k}}: \CC \Si_k \rightarrow  \ParAlg$ is the inclusion map, and so it may be seen as the
partition algebra analogue of the {\it spherical Cherednik algebra}, 
considered for example in \cite{Rouquier}. For any $t \in \CC$, 
there is a specialization map 
$ x \mapsto t $ for $ \SpheAlg $ and we denote by $  \SpheAlg(t) $ the corresponding 
specialized algebra. 
We show that the $  \SpheAlg(t) $'s, for $ t $ running over $ \CC$, are algebras of fundamental
interest in the representation theory of diagram algebras, and even beyond that.  

\medskip
A first main result of our paper, given in Theorem \ref{lemma2}, 
is the determination
of the rank $  \rank_{\CC[x]} \,  \SpheAlg $ of $ \SpheAlg $. We find 
that $  \rank_{\CC[x]} \,  \SpheAlg = bp_k $ where $ bp_k $ is the
cardinality of {\it bipartite partitions} $ \BiPar_k $ of $k$.
Bipartite partitions are classical combinatorial objects whose 
history goes back more than a century to the work of Macmahon and others, see for example \cite{MacMahon}.
The sequence $ (bp_0, bp_1, bp_2, bp_3,  bp_4, bp_5, \ldots ) = (1, 2, 9, 31, 109, 339, \ldots ) $
is A002774 in the On-Line Encyclopedia of Integer Sequences.
%{\color{black}{\sout{see \url{https://oeis.org/A002774}.}}}

\medskip
Let $ V {\color{black}{=V_n}} $ be a vector space of dimension $n$. 
Then the second named author and V. Jones
independently proved that $ \ParAlg(n)$ is in Schur-Weyl duality with the group algebra
$ \CC \Si_n $, acting 
diagonally on $ V_n^{\otimes k } $, see \cite{Jones}, \cite{PMartin}. 
This is an important {\it double centralizer property} that allows us to pass
representation{\color{black}{-}}theoretic information back and forth between the 
module categories for $ \CC \Si_n $ and $ \ParAlg(n)$. 

\medskip
A main motivation for our work is to establish an analogous
{\it double centralizer property} involving $ \SpheAlg(n) $ and $ \CC \Si_n $, but this time with
commuting actions 
on the symmetric power space $ S^k V_n $, and to study some of its consequences.
We achieve this goal in section \ref{Schur-Weyl duality}
of our paper, culminating in our Theorems \ref{teorem 6} and \ref{joining the results}. 
We obtain an isomorphism of $ (\CC \Si_n, \SpheAlg(n)) $-bimodules
\begin{equation}\label{symmetricpower}
S^k V_n \cong \bigoplus_{ \lambda \in  \ParSph} S(\lambda) \otimes G_k(\lambda) 
\end{equation}
where $ S(\lambda) $ is the Specht module for $ \CC \Si_n $ and $ G_k(\lambda)  $ is 
a simple module for $ \SpheAlg(n) $, 
for $ \ParSph \subseteq \Par_k $ 
a concretely defined subset of the set of partitions of $ k $, see \eqref{defineLambaSph}. 
Moreover, for $ \lambda \in \ParSph $
we obtain an explicit dimension formula
\begin{equation}\label{keyingredient}
 \dim G_k(\lambda) =
    \sum_{ \nu \in \Par_k^{\le n} } {  K_{\lambda,  \Phi(\nu)} }
\end{equation}
where $  K_{\lambda,  \Phi(\nu)} $ is the Kostka number, and $ \Phi:  \Par_k^{\le n} \rightarrow \Par_n $
is a \lq multiplicity\rq\ function defined on the partitions of $ k $ of length less than $ n $.
A key ingredient in the proof of \eqref{keyingredient} is a direct sum decomposition of $  S^k V_n $
in terms of {\it permutation modules} for $ \CC \Si_n $. A related version of this decomposition
was obtained by Harman in \cite{Harman}, 
but for the reader's convenience we provide its (simple) proof 
in Theorem \ref{firstthm}.

\medskip
%{\color{black}{\sout{Considering small value of $ n$}}}
It is clear, {\color{black}{however,}} that the $ G_k(\lambda) $'s do not exhaust
all the simple $ \SpheAlg(n)$-modules, 
and therefore we embark on 
a systematic study of the representation theory of $ \SpheAlg(t)$,
for all specializations $ x \mapsto t \in \CC $ (except $ t=0 $ that we sometimes omit for
brevity of the presentation). We find that
$ \SpheAlg(t)$ is semisimple if $ t \not\in \{0,1,2,\ldots, 2k-2 \} $
but not if $ t \in \{1,2,\ldots, 2k-2 \} $
(although 
$ S^k V_n $ is always a semisimple $ \SpheAlg(n)$-module, as can be read off from 
\eqref{symmetricpower}).

\medskip
A main ingredient in our study is  
the fact, shown for example by \ppmG{K\"onig},
Xi and Doran-Wales, that $ \ParAlg$ is 
a {\it cellular algebra} in the sense of Graham and Lehrer, 
and therefore 
$ \SpheAlg$ is a cellular algebra too, being an idempotent truncation
of a cellular algebra, see eg. \cite{DW}, \cite{GL}, \cite{KoXi} and \cite{Xi}.
The cell modules are of the form $ e_k \Delta_k(\lambda) $ where
the $ \Delta_k(\lambda)$'s are cell modules for $ \ParAlg$ and
$ \lambda \in \Lambda^k  = \bigcup_{l=0}^{k} \Par_l $.

\medskip
In section 
\ref{sectioncellularity}, we combine results due to Murphy, see \cite{Murphy},
with specific diagrammatic calculations in order to obtain 
a basis for $ e_k \Delta_k(\lambda) $. 
In particular, in Theorem \ref{stateandproveA} we show 
\begin{equation}
\dim e_k \Delta_k(\lambda) =
    \sum_{i=l}^k \sum_{\substack{  \nu \in \Par_i  \\  \Psi(\nu) \in \Par_l }} K_{\lambda,  \Psi(\nu)} |  \Par_{{k-i}} |
\end{equation}
where $ \Psi:  \Par_i \rightarrow  \bigcup_{ k=0}^{\infty} \Par_k$ is a new multiplicity function. 
In Corollary \ref{sphpar}, we deduce from this that $ e_k \Delta_k(\lambda) \neq 0 $ if and only if $ \lambda \in \Lambda_{sph}^{k}$
where $ \Lambda_{sph}^{k} \subseteq \Lambda^k$ is another concretely defined subset of $ \Lambda^k$, see
\eqref{concretelydefined2}.

\medskip
The set $ \Lambda_{sph}^{k} $ is the natural index set for the representation theory of
$ \SpheAlg(t) $.
\ppm{It follows immediately from the construction that
  $ ( k ) \in \Lambda_{sph}^{k} $
  whereas $ ( 1^k ) \notin \Lambda_{sph}^{k} $ if $ k\ge 2 $.
  However the primitive idempotents associated with these cell modules
  are of  %independent
  intrinsic interest, so let us
(upon the suggestion of the referee) 
  add a remark in this direction.}
{\color{black}{For $ n \ge 2k $
the primitive idempotents 
associated with the cell modules for 
$ \ParAlg(n) $ 
were
determined by the second named author and Woodcock in \cite{MartinWoodcock2}. 
Recently,  Benkart-Halverson and Campbell found cancellation-free
expressions for the idempotents associated with
the cell modules $ \Delta_k (k) $ and $ \Delta_k( 1^k ) $, see \cite{BH} and \cite{Campbell}. 
%Using this, we get
\ppm{The forms of these expressions are indeed of intrinsic interest, and
  of course they verify}
that 
$ ( k ) \in \Lambda_{sph}^{k} $
whereas $ ( 1^k ) \notin \Lambda_{sph}^{k} $ if $ k\ge 2 $. }}

\medskip
The 
simple $ \SpheAlg(t) $-modules  
are $ \{ e_k L_k(\lambda) \, | \, \lambda \in \Lambda_{sph}^{k} \} $,
obtained via multiplication with $ e_k $ on the simple $ \ParAlg(t)$-modules $ L_k(\lambda) $. 
This, combined with results by the second named author, see
\cite{PMartin1}, leads
to our main Theorem \ref{mainThm} that describes  
the decomposition numbers and dimensions of the simple modules for $ \SpheAlg(t) $, 
in all cases except $ t=0$.

\medskip
Our proofs rely heavily on
Corollary \ref{quasiheriditarySphe},
stating that $ \SpheAlg(t)$ is a quasi-hereditary algebra when $ t \neq 0 $.
{\color{black}{This may be regarded as a key property of $ \SpheAlg(t)$, and also $ \ParAlg(t)$ has this
property.}}
In the final section \ref{tilting} of the paper, we take the opportunity to determine the
indecomposable projective modules and the indecomposable tilting modules for both algebras. 

\medskip
In \cite{NaPaulSriva}, S. Narayanan, D. Paul and S. Srivastava introduced the {\it multiset partition algebra}
$ \mathcal{MP}_k $ via an explicit combinatorial definition of its structure coefficients. It was
further generalized and studied by R. Orellana and M. Zabrocki in \cite{OrellanaZabrocki2} and
by A. Wilson in \cite{Wilson}. 
In \cite{NaPaulSriva} it was proved that $ \mathcal{MP}_k(n) $ is in Schur-Weyl 
duality with $ \mathbb{C} \Si_n $ on $ S^k V_n $ and so it follows that 
$ \SpheAlg(n) $ and $ \mathcal{MP}_k(n) $ are isomorphic when $ n \ge 2k $. Actually
A. Wilson has kindly informed us of a(n unpublished) proof showing that $ \SpheAlg(t) $ and $ \mathcal{MP}_k(t) $
are isomorphic in general. Given this, it is likely that the simple modules
for $ \mathcal{MP}_k(n) $ that are described in \cite{NaPaulSriva}, \cite{OrellanaZabrocki2} and \cite{Wilson} in 
terms of
{\it semistandard multiset tableaux}, are the $ G_k(\lambda)$'s of the present paper. In this sense, the Schur-Weyl duality results of our section
\ref{Schur-Weyl duality} may be considered as a complimentary approach to some of the results for $ \mathcal{MP}_k(n) $, 
developed in \cite{NaPaulSriva}, \cite{OrellanaZabrocki2} and \cite{Wilson}. On the other hand, our main results 
in sections \ref{sectioncellularity}, \ref{decompositionnumbers} and \ref{tilting},
for example the complete classification of the simple modules for $ \SpheAlg(n) $, 
and the description of these modules in 
Theorem \ref{mainThm}, have not been obtained in the 
$\mathcal{MP}_k(n) $-setting.

\medskip
Let us give a brief overview of the organization of the paper. In the following section
\ref{Basic notation} we fix the basic notation to be used throughout the paper. This
concerns integer partitions, Young tableaux and other concepts related to the representation
theory of the symmetric group. In section \ref{Bipartite partitions} we recall the notion
of bipartite partitions $ \BiPar_k $ and introduce the corresponding diagrammatic representations.
For $ b \in \BiPar_k$, we further recall the lexicographic normal form $ \norm(b) $ and Garsia and Gessel's
normal form $GG(b)$ from \cite{GG}. In section \ref{The partition algebra} we recall the partition
algebra $ \ParAlg$ and introduce the spherical partition algebra $ \SpheAlg$, the main protagonist of
our paper. In section \ref{Rank of the} we show that $ \rank_{\CC[x]} \SpheAlg = bp_k $,
by constructing a concrete basis for $ \SpheAlg$. This uses diagrammatic arguments
involving the normal form $\norm(b)$ from section \ref{Bipartite partitions}. 
In section \ref{Schur-Weyl duality} we construct commuting actions of $ \SpheAlg(n) $ and $ \CC \Si_n $
on $ S^k V_n $ and show that they satisfy a double centralizer property. Motivated by
a recent paper of Benkart, Halverson and Harman, see \cite{BHH}, 
we find a direct sum decomposition of $  S^k V_n $ 
in terms of permutation modules for $ \CC \Si_n $, which allows us to determine 
the dimension of the irreducible $ \SpheAlg(n) $-modules $ G_k(\lambda) $ that appear in $  S^k V_n $. 

\medskip
%{\color{black}{\sout{Section \ref{sectioncellularity} is the most technically involved section of our paper. }}}
In section \ref{sectioncellularity} we show that $ \SpheAlg $ is a cellular algebra, determine a basis for its cell modules
and determine the parametrizing poset for $ \SpheAlg $. Using Garsia and Gessel's normal form
for bipartite partitions, we further construct a Robinson-Schensted-Knuth type bijection for 
$ \SpheAlg $. Finally, we show that $ \SpheAlg(t) $ is quasi-hereditary when $ t \neq 0$. 

\medskip
In section \ref{decompositionnumbers} we obtain the main Theorems involving the decomposition
numbers for $ \SpheAlg(t)$ and finally, in section \ref{tilting}, we obtain the Loewy structure
of the indecomposable projective modules and tilting modules.

\subsection*{{\color{black}{Acknowledgements}}}
The authors would like to express their gratitude to Stephen Griffeth and Luc Lapointe for 
insightful comments on the results of the paper. They are especially grateful to Alexander Wilson for
pointing out the connection to the multiset partition algebra and for mentioning the references
\cite{NaPaulSriva}, \cite{OrellanaZabrocki2} and \cite{Wilson}.   
Finally, they wish to express their gratitude to the anonymous referee
{\color{black}{for indicating several research projects related to the present work
and for communicating detailed suggestions that helped them improve substantially the clarity and presentation
of the paper.}}

\section{Basic notation}\label{Basic notation}
In this section we quickly fix the relevant notation concerning partitions, Young tableaux, and so on.

\medskip
For $ k \in \N$ we let $\Par_k $ be the set of {\it integer partitions} of $ k $, that is 
weakly decreasing
positive integer sequences $ \lambda=(\lambda_1, \lambda_2, \ldots, \lambda_p ) $
such that $ \lambda_1+\lambda_2+\ldots + \lambda_p = k$.
The length of $ \lambda $ is defined to be $ \ell(\lambda) = p $ and its order is defined as $ | \lambda | = k $. 
The set of partitions in $ \Par_k $ of length less then or equal to $ l $ is denoted
$  \Par_k^{\le l } $ and
the set of partitions in $ \Par_k $ of length equal to $ l $ is denoted
$  \Par_k^{l } $. 
Using the convention $ \Par_0 = \emptyset $ we define
$ \Par= \bigcup_{ k=0}^{\infty} \Par_k $ and $ \Par^{\le l }= \bigcup_{ k=0}^{\infty} \Par_k^{\le l } $. 
We sometimes write $ \lambda \in \Par_k $ in the form $ \lambda=
(\lambda_1^{a_1}, \lambda_2^{a_1}, \ldots, \lambda_p^{a_p} ) $ where $ \lambda_1 >
\lambda_1 >  \ldots > \lambda_p  $ and where $ a_i$ is the multiplicity of $ \lambda_i $ in $ \lambda$.

\medskip
More generally, for $ k \in \N$ we let $\Comp_k $ be the set of {\it compositions of $k$},
that is positive integer sequences $ \lambda=(\lambda_1, \lambda_2, \ldots, \lambda_p ) $
such that $ \lambda_1+\lambda_2+\ldots + \lambda_p = k$.
For $ \mu = (\mu_1, \mu_2, \ldots, \mu_p ) \in \Comp_k $ and $ \nu = (\nu_1, \nu_2, \ldots, \nu_q ) \in \Comp_l $ 
we define $ \mu \cdot \nu = (\mu_1, \mu_2, \ldots, \mu_p, \nu_1, \nu_2, \ldots, \nu_q ) \in \Comp_{k+l} $.
For $ \mu =  (\mu_1, \mu_2, \ldots, \mu_p ) \in \Comp_k $
we let $ \ord(\mu) \in \Par_k $ be the partition obtained
from $ \mu $ 
by reordering the $ \mu_i$'s. 

\medskip
We identify $ \lambda \in \Par_k $, and more generally
$ \lambda \in \Comp_k $, 
with its Young diagram, for example
\begin{equation}\label{dib14}
  (5,3,2) =  \raisebox{-.4\height}{\includegraphics[scale=0.7]{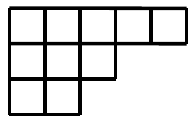}}
 \raisebox{-13\height}{,}  
  \, \, \, \, \, \, 
(2,3,5) =  \raisebox{-.4\height}{\includegraphics[scale=0.7]{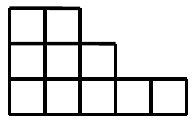}} \, \, 
\raisebox{-13\height}{.}
\end{equation}
We use matrix convention to label the boxes, also called nodes, of $ \lambda $.
Thus, $(1,1),(1,2), \ldots, (1,\lambda_1) $
are the nodes of the first row of $\lambda$, etc.
We write $ u \in \lambda$ if $ u $ is a node of $ \lambda $. 
For $ \lambda \in \Comp_k$, a $\lambda$-{\it tableau}
$ \s$ is
a filling of the nodes of $ \lambda$ with the numbers $ \{1,2,\ldots, k\} $, each number occurring exactly once.
A $\lambda$-tableau $ \s$, is called row/column standard if the numbers in each row/column are
increasing from left to right/top to bottom, and is called standard if it is both row and column standard.
The set of all $\lambda$-tableaux is denoted $ \tab(\lambda)$ 
and the set of all standard $\lambda$-tableaux is denoted $ \std(\lambda)$.
For $ \s \in \tab(\lambda)$ we define $ \shape(\s) = \lambda$. 
Below are examples
of a row standard and a standard $\lambda $-tableau, for $ \lambda = (5,3,2) $. 

\begin{equation}\label{dib14+15}
  \raisebox{-.4\height}{\includegraphics[scale=0.7]{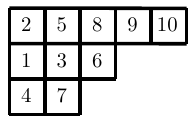}}
 \raisebox{-13\height}{,}
  \, \, \, \, \, \, \, \, \, 
 \raisebox{-.4\height}{\includegraphics[scale=0.7]{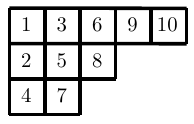}}
 \raisebox{-13\height}{.}
\end{equation}

Suppose that $ \lambda \in \Comp_k $. 
For $ \s, \T \in \tab(\lambda) $ we write $ \s \sim \T$ if $ \s$ can be obtained from $\T$
by permuting the numbers within the rows of $ \T $. This defines an equivalence relation
on $ \tab(\lambda)$. The equivalence classes under $ \sim $ are called $\lambda$-{\it tabloids} and
the tabloid represented by $ \T $ is denoted $ \{ \T \} $. We let $\{ \tab(\lambda) \} $ denote
the set of $ \lambda$-tabloids.

\medskip
Let 
$ \Si_k $ be the symmetric group on $ \{1, 2 , \ldots, k \} $ and let $ \lambda \in \Par_k$. 
Then there is a natural left $ \Si_k $-action on $\tab(\lambda) $, with $ \sigma \in \Si_k $ acting on
the entries of $\s \in \tab(\lambda)$. For example, if
$ \lambda = (4,3,2) $ and 
$ \sigma = (1,3,2) (4,6,5) $ in cycle notation, then

\begin{equation}\label{dib17}
 \sigma\left ( \raisebox{-.4\height}{\includegraphics[scale=0.7]{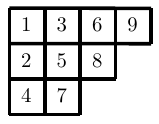}}\right) = 
 \raisebox{-.4\height}{\includegraphics[scale=0.7]{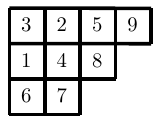}}
 \raisebox{-13\height}{.}
\end{equation}

Note that $\tab(\lambda) \cong \Si_k $, where $ \Si_k $ acts on $\Si_k $
via left multiplication. Note also 
that the left $ \Si_k $-action on $\tab(\lambda) $ induces a left $ \Si_k $-action on  $ \{ \tab(\lambda) \}$. 
Let $ M(\lambda) $ be the free $ \CC$-vector space
on $ \{ \tab(\lambda) \} $. 
Then the left $ \Si_k$-action on $ \{ \tab(\lambda) \} $ gives rise to 
a left $ \CC \Si_k$-module structure on $ M(\lambda) $. This is the {\it permutation module}
for $ \CC \Si_k$.  We have 
\begin{equation}\label{multinomial}
\dim  M(\lambda) = \binom{k}{ \lambda_1,  \lambda_2,  \ldots ,\lambda_p} 
\mbox{ where }  \binom{k}{ \lambda_1,  \lambda_2,  \ldots ,\lambda_p} 
\mbox{ is the multinomial coefficient. }
\end{equation}

\medskip
The irreducible $ \CC \Si_k $-modules are the {\it Specht modules}
$ \{ S(\lambda)  \, | \,  \lambda \in \Par_k\}$, see for example \cite{James}. 
We have $ \dim S(\lambda) = \Std(\lambda) $.

\medskip
Let $ \lambda \in \Par_k $ and let $ \mu =(\mu_1, \mu_2. \ldots, \mu_q) \in \Comp_k$. 
Then a {\it semistandard} $ \lambda$-tableau 
$ \s$ of type $ \mu $ is
a filling of the nodes of $ \lambda$, with
the number $1$ occurring 
$ \mu_1 $ times, 
the number $2$ occurring 
$ \mu_2 $ times and so on,  
such that the numbers in each row of $ \s $ are weakly increasing from left to right,
whereas the numbers in each column of $ \s $ are strictly increasing from top to bottom.
For example, if $ \lambda = (4,3, 2) $ and $ \mu = (3,3,3) $, the following are
the two possible semistandard $ \lambda$-tableaux of type $ \mu$

\begin{equation}\label{dib19}
  \raisebox{-.4\height}{\includegraphics[scale=0.7]{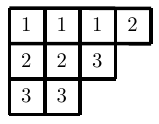}}
 \raisebox{-13\height}{,}
  \, \, \,   \, \, \,   
 \raisebox{-.4\height}{\includegraphics[scale=0.7]{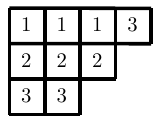}}
 \raisebox{-13\height}{.}
\end{equation}
The set of semistandard $ \lambda$-tableaux of type $ \mu$ is denoted $ \sstd(\lambda, \mu ) $
and its cardinality $ |  \sstd(\lambda, \mu ) | $ is the {\it Kostka number} $ K_{\lambda \mu} $. 
For example, 
for $ \lambda= (4,3,2) $ and $ \mu = (3,3,3) $ we have $ K_{\lambda  \mu} =2$, 
as can be read off from \eqref{dib19}.

\medskip
Let $  [M(\mu): S(\lambda) ]   $ be the multiplicity of $ S(\lambda ) $ in $ M(\mu) $.
Then we have that 
\begin{equation}\label{multiKostka}
 [M(\mu): S(\lambda) ] = K_{\lambda \mu}. 
\end{equation}
If $ \nu = (\nu_1, \ldots, \nu_l) \in \Comp_k $ is obtained from $ \mu = 
(\mu_1, \ldots, \mu_l) \in \Comp_k $ via $ \nu_i = \mu_{ \sigma(i) } $ for some
$ \sigma \in \Si_l $, then there exists an
isomorphism $  M(\mu) \cong M(\nu) $ of $ \CC \Si_k $-modules. This is reflected in 
\eqref{multiKostka}, since $ K_{\lambda \mu}  = K_{\lambda \nu}  $ in that case.

\section{Bipartite partitions}\label{Bipartite partitions}
In this section we recall the notion of
bipartite partitions and introduce the spherical partition algebra. 

\medskip
  
For $ k \in \N$, we let $ \BiPar_k $ be the set of {\it bipartite} partitions of $ k  $. That is, 
$ \BiPar_k $ is 
the set of multisets
$ b= \{ [x_1, y_1],  [x_2, y_2], \ldots,  [x_a, y_a]   \} $
of pairs $ [x_i, y_i] $ such that 
$ x_i $ and $ y_i $ are nonnegative integers, not both zero, 
satisfying 
\begin{equation}
\sum_{i =1}^a x_i = \sum_{i =1 }^a  y_i= k. 
\end{equation}
Let $ bp_k $ be the cardinality of $  \BiPar_k $. Then $bp_1= 2 $, since 
$ \BiPar_1 $ consists of the multisets
\begin{equation}
 \{ [1, 1] \},  \{ [1, 0], [0,1] \}.
\end{equation}
Similarly,
$bp_2= 9 $, since 
$ \BiPar_2 $ consists of the multisets
\begin{equation}
  \begin{array}{l}
    \{ [2,2] \}, \,  \{ [1, 0], [1,2] \}, \,  \{ [2, 1], [0,1] \},\,  \{ [1, 1], [1,1]\},\,  \{ [2, 0], [0,2] \},\, 
 \{ [2, 0], [0,1], [0,1] \}
    \\
 \{ [1, 0], [1,0], [0,2] \}, \,  \{ [1, 1], [1,0], [0,1] \}, \, \{ [1, 0], [1,0], [0,1], [0,1] \}. 
  \end{array}
\end{equation}
We use the convention that $bp_0= 1 $. The sequence
\begin{equation} ( bp_0, bp_1, bp_2, bp_3,  bp_4, bp_5, \ldots )
= (1, 2, 9, 31, 109, 339, \ldots )
\end{equation}
is A002774 in the OEIS.
%{\color{black}{\sout{see \url{https://oeis.org/A002774}.}}}

\medskip
Bipartite partitions in $ \BiPar_k $ are also known as {\it vector} partitions of $ [k,k]$. Their
history goes back to the work of Macmahon, and their combinatorics have 
been studied for example in \cite{F. C. Auluck}, \cite{GG} and \cite{MacMahon}.
%{\color{black}{\sout{In the present paper, we shall see that they are also the dimensions of certain interesting algebras. }}}

\medskip
For $ b= \{ [x_1, y_1], [x_2, y_2], \ldots,  [x_a, y_a]  \} \in \BiPar_k $ we represent
each part $ [x_i, y_i] $ of $ b $ via two parallel horizontal lines of points, the
top row containing $ x_i $ points and the bottom row containing $ y_i $ points, that are joined
via a \emph{propagating} line from the leftmost top point to the leftmost bottom point, for example
\begin{equation}
  [5,3] =  \,  \raisebox{-.4\height}{\includegraphics[scale=0.8]{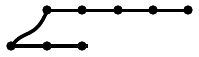}}
\raisebox{-6\height}{.}
 %\raisebox{-6\height}{,} \,  \, \, \, \, \, \, \, \, \, \, \, 
%[5,4] = \,   \raisebox{-.4\height}{\includegraphics[scale=0.8]{dib2.pdf}} \, \, 
\end{equation}
We represent $ b $ itself diagrammatically by concatenating the diagrams of the parts $ [x_i, y_i] $ from
left to right, for example for $ b = \{[3, 1], [2, 2], [3, 2], [0, 4], [2, 1] \} $ we have 
\begin{equation}\label{dib3}
b \mapsto    \raisebox{-.4\height}{\includegraphics[scale=0.8]{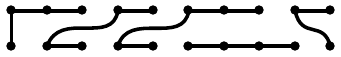}}
\raisebox{-6\height}{.}
\end{equation}
Note that since elements of $ \BiPar_k $ are multisets, 
this diagrammatic representation of $ b \in \BiPar_k $ is not unique, since any permutation of 
the parts of $ b \in \BiPar_k $ does not change $ b $. For example we have 
\begin{equation}
\{ [2,1], [1,2] \}  \mapsto  \,  \raisebox{-.4\height}{\includegraphics[scale=0.8]{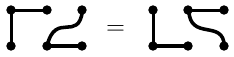}}
\raisebox{-6\height}{.}
\end{equation}
In order to remediate this nonuniqueness, we introduce for
$ b \in  \BiPar_k $ 
the {\it normal form} $ \norm(b) $, using the appropriate lexicographic order. To be precise,
suppose that $ b = \{ [x_1,y_1], [x_2,y_2], \ldots, [x_a , y_a ] \} $. Then we define
$ \norm(b) = \left( [x_{\sigma(1)},y_{\sigma(1)}], [x_{\sigma(2)},y_{\sigma(2)}], \ldots, [x_{\sigma(a)}, y_{\sigma(a)}]
\right) $
where $ \sigma \in \Si_a $ is chosen such that if $ i \ge j $ then either
$ x_{\sigma(i)} <  x_{\sigma(j )} $ or ($ x_{\sigma(i)} =  x_{\sigma(j )} $ and
$ y_{\sigma(i)} \le  y_{\sigma(j )} $). For example, we have
\begin{equation}\label{lexi}
  \norm \big( \{ [1,2], [2,1], [4,1], [0,2], [0,1], [1,2], [1,1], [3,2]\} \big)   =
  \big( [4,1], [3,2], [2,1], [1,2], [1,2], [1,1], [0,2], [0,1] \big) .   
\end{equation}
Using the normal form $ \norm(b) $, elements of $ \BiPar_k $ may be viewed as sequences of pairs $ [x_i, y_i] $
rather than multisets of such pairs. For $ \norm(b) $ applied to $ b $ as in \eqref{dib3}
we have
\begin{equation}\label{dib11}
  \norm(b) \mapsto  \raisebox{-.4\height}{\includegraphics[scale=0.8]{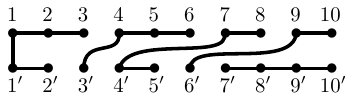}}
  \raisebox{-6\height}{.}
\end{equation}

\medskip
In \cite{GG}, Garsia and Gessel gave another characterization of $ \BiPar_k $,
that we shall need. Let $ \lambda = (\lambda_1, \lambda_2, \ldots, \lambda_l) \in \Par_k $
and $ \sigma= (\sigma_1, \sigma_2, \ldots, \sigma_l ) \in \Si_l $ be a symmetric group
element written in {\it  permutation notation}, by which we mean that $ \sigma_i \in \{ 1,2, \ldots, l \} $ and
that 
$ \sigma $ maps $ i $ to $ \sigma_i $ for all $ i $.  Then $ \lambda $ is said to be {\it $ \sigma$-compatible} if
$ \lambda_i = \lambda_{i+1} $ implies $ \sigma_i < \sigma_{i+1} $. 

Suppose now that $ b = 
 \{ [x_1,y_1], [x_2,y_2], \ldots, [x_a , y_a ] \}
\in  \BiPar_k $ and consider a diagrammatic representation for $ b $ as in
\eqref{dib3}.
Define $ \lambda^{top} $ as the partition obtained from the nonzero $ x_i$'s via reordering, and define 
similarly $ \lambda^{bot} $. 
Next reorder the top points and bottom points of the diagram in such a way that
there are no crossings between the propagating
lines leaving 
parts of the same length in $ \lambda^{top} $, and similarly for $ \lambda^{bot} $, 
and let $ GG(b) $ be the resulting diagram. 
Define $ \lambda^{top, pro} $ to 
be the partition extracted from $ \lambda^{top} $ by eliminating the parts
with no propagating lines, and define similarly $ \lambda^{bot, pro} $. 
Then $ \lambda^{top, pro} $ and $ \lambda^{bot, pro} $ are partitions of the same length, say $ l$, 
and so we may define $ \sigma =( \sigma_1, \sigma_2, \ldots, \sigma_l ) \in \Si_l $
by the condition that 
$ \lambda^{bot, pro}_{ 1} $ is connected to $ \lambda^{top, pro}_{ \sigma_1} $, whereas 
$ \lambda^{bot, pro}_{ 2} $ is connected to $ \lambda^{top, pro}_{ \sigma_2} $, and so on.
With this notation, Theorem 2.1 of \cite{GG} states that 
$  \lambda^{bot, pro} $ is $\sigma$-compatible
whereas 
$  \lambda^{top, pro} $ is $\sigma^{-1}$-compatible, and that $ \BiPar_k $ is characterised by
these properties. In other words, the diagram $ GG(b) $ is another normal form for $ b \in \BiPar_k$.
For example, for $ b $ as in \eqref{dib3}, we have 
\begin{equation}\label{gg}
GG(b)  \mapsto  \,  \raisebox{-.4\height}{\includegraphics[scale=0.8]{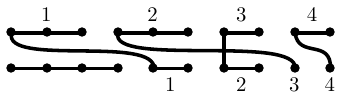}}
\end{equation}
and so $ \lambda^{top, prop} = (3,3,2,2) $, $ \lambda^{bot, prop} = (2,2,1,1) $
and $ \sigma= (1,3,2,4)$.

We define the {\it propagating part} of $ GG(b) $ to be the diagram
obtained from $ GG(b) $ by removing 
all components that are completely contained in the top line or in the bottom line of points. For example,
for $ GG(b) $ as in \eqref{gg}, the propagating part is 
\begin{equation}\label{ggg}
  \raisebox{-.4\height}{\includegraphics[scale=0.8]{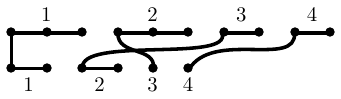}}
\raisebox{-10\height}{.}
\end{equation}

\section{The partition algebra and the spherical partition algebra}\label{The partition algebra}
We next recall the partition algebra $ \ParAlg$; it was 
introduced by the second named author via considerations
in statistical mechanics, see \cite{PMartin}.
Let $ \SetPar_k $ be the set of set partitions on $ \{ 1,2 \ldots, k \} $, 
that is the set of equivalence relations
$ d $ on $ \{ 1,2 \ldots, k \}$. For even subscript $ 2 k $ we
shall usually think of $ \SetPar_{2k} $ as set partitions on $ \{ 1,2 \ldots, k \} \cup 
\{ 1^{\prime},2^{\prime} \ldots, k^{\prime} \}$. 
If $ d \in \SetPar_{k} $ we
write $ d = \{ d_1, d_2, \ldots, d_a   \} $ where the $ d_i $'s are the classes, or {\it blocks}, of $ d $.
If furthermore $ d \in \SetPar_{2k}$, we represent $d$ diagrammatically using two parallel
horizontal lines of points,
just as for elements of $\BiPar_k$, but this time labeling the top points
$ \{ 1, 2, \ldots, k\} $ and the bottom points $ \{ 1^{\prime}, 2^{\prime}, \ldots, k^{\prime}\} $,
from left to right. We draw lines between these points in such a way that the connected components,
in {\color{black}{the graph-theoretic}} sense, 
of the corresponding graph are exactly the blocks of $ d $, for example
\begin{equation}\label{dib5}
  \{ \{1\}, \{2,3, 7,8,9, 6^{\prime}, 7^{\prime}, 8^{\prime}\}, \{ 4,5,6,1^{\prime}, 2^{\prime}\},
\{ 3^{\prime}, 4^{\prime}, 5^{\prime} \}, 
  \{ 9^{\prime} \}
     \} \,  \mapsto  \,  \raisebox{-.4\height}{\includegraphics[scale=0.8]{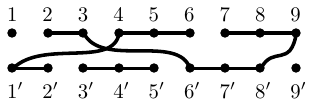}}
\raisebox{-4.3\height}{.}
\end{equation}
Note that, just as for elements of $\BiPar_k$, this diagrammatic representation of $ d \in \SetPar_{2k} $ is
not unique.

For $ d = \{ d_1, d_2, \ldots, d_a \} \in \SetPar_{2k} $, we say that a block $ d_i $ is {\it propagating}
if $ d_i \cap \{ 1,2, \ldots, k  \} \neq \emptyset $ and 
$ d_i \cap \{ 1^{\prime},2^{\prime}, \ldots, k^{\prime}  \} \neq \emptyset $. 
If $ d_i \cap \{ 1,2, \ldots, k  \} \neq \emptyset $ we say that $ d_i \cap \{ 1,2, \ldots, k  \}$
is an \emph{intersection top block} for
$d $ and if $ d_i \cap \{ 1^{\prime},2^{\prime}, \ldots, k^{\prime}  \} \neq \emptyset $
we say that $ d_i \cap \{ 1^{\prime},2^{\prime}, \ldots, k^{\prime}  \} $ is an \emph{intersection
  bottom block} for $d $.

\medskip
{\color{black}{We define $ \ParAlg$ as the $ \CC[x]$-algebra that, as a $ \CC[x]$-module,}} is free on  
$ \SetPar_{2k} $, and that has multiplication defined as follows. For elements
$ d, d_1 \in \SetPar_{2k} $, let $ d \circ_1 d_1 $ be the concatenation 
of $ d $ and $ d_1 $ with $ d$ on top of $ d_1$. There may be one or several \lq internal\rq\ connected components
of $ d \circ_1 d_1 $, 
that is components that do not intersect any of the top or bottom points of $ d \circ_1 d_1 $. 
Let $ d \circ_2 d_1 $
be the diagram obtained from $ d \circ_1 d_1 $ by removing these $ N $, say, internal 
components. There may still one or several \lq internal points\rq\ of $ d \circ_2 d_1 $, that is points that are
neither top or bottom points of $ d \circ_2 d_1 $, and we let $ d \circ_3 d_1 $
be the diagram obtained from $ d \circ_2 d_1 $ by eliminating these points. 
We may now view $ d \circ_3 d_1 $ as the diagram of a set partition and 
the product in $ \ParAlg $ of $ d $ and $d_1 $ is defined as $ d d_1 = x^N d \circ_3 d_1 $. 
The product of two general elements of $ \ParAlg $ is defined
{\color{black}{by the linear extension of the multiplicative operation we have defined.}} 

\medskip
For example, if
\begin{equation}\label{components}
  d=  \raisebox{-.45\height}{\includegraphics[scale=0.8]{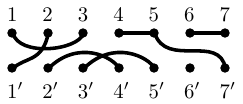}},\, \,\, \,\, \,
  d_1=  \raisebox{-.45\height}{\includegraphics[scale=0.8]{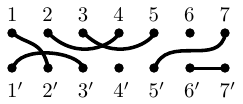}}
\end{equation}
we have that 
\begin{equation}
  d d_1 =  \raisebox{-.45\height}{\includegraphics[scale=0.8]{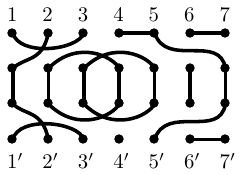}}=\, \,\, 
     x^3 \,  \raisebox{-.45\height}{\includegraphics[scale=0.8]{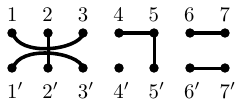}}
\raisebox{-6\height}{.}
\end{equation}
One checks that this rule gives rise to a well-defined associative 
multiplication on $  \ParAlg$, in other words, $ d d_1 $ does not depend on the choices of
diagrammatic representations for $ d $ and $ d_1 $. 

\medskip

For any $ t \in \CC $ we define the specialized partition algebra 
$\ParAlg(t) =  \ParAlg \otimes_{\CC[x] } \CC $ where $ \CC $ is made into an $ \CC[x] $-algebra
via $ x \mapsto t $. 

\medskip
As is well known, $ \Si_k $ is a Coxeter group on generators $ S = \{s_1, s_2, \ldots, s_{k-1} \} $
where $ s_i $ is the simple transposition $ s_i = (i,i+1) $.
Let $\CC[x] \Si_k $ be the group algebra for $ \Si_k $ over $ \CC[x] $. Then 
there is a natural algebra inclusion $ \iota_k:  \CC[x] \Si_k \xhookrightarrow{} \ParAlg $
given by 
\begin{equation}\label{dib10}
  s_i  \mapsto  \raisebox{-.45\height}{\includegraphics[scale=0.8]{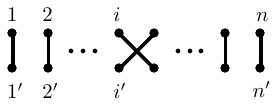}}
\raisebox{-5\height}{.}
\end{equation}
Let $ \displaystyle e_k = \iota_k\left ( \dfrac{1}{k! } \sum_{ \sigma \in \Si_k } \sigma\right)  $. 
Then $ e_k $ is an idempotent of $ \ParAlg $. We use it to introduce the protagonist
of the present paper. 
\begin{definition}
  The spherical partition algebra $\SpheAlg $ is defined as the idempotent truncation
  of $ \ParAlg $ with idempotent $ e_k $, that is
\begin{equation}
 \SpheAlg = e_k \ParAlg e_k. 
\end{equation}
Similarly, for $ t \in \CC $ we define the specialized spherical partition algebra
$\SpheAlg(t) $ as $\SpheAlg(t) = e_k \ParAlg(t) e_k  $. 
 \end{definition} 
Note that $\SpheAlg $ is a subalgebra of $ \ParAlg $, but not a unital subalgebra, since
the one-element for $\SpheAlg $ is $ e_k $, and similarly for $ \SpheAlg(t) $. 
  
\section{Rank of the Spherical Partition Algebra}\label{Rank of the}
As a $ \CC[x]$-module $ \SpheAlg $ is automatically free,
since $ \CC[x]$ is a PID 
and $ \SpheAlg $ is 
a submodule of the free $ \CC[x]$-module $ \ParAlg $, and hence torsion-free. 
Our first task is to determine the rank of $ \SpheAlg $. 

\medskip
For this, we first observe that any diagrammatic
representation of $ b \in \BiPar_k $ may be viewed 
as an element of 
$\SetPar_{2k} $. For example, for $ b  $ as in \eqref{dib3}, 
and hence $ \norm(b) $ as in \eqref{dib11}, we have
\begin{equation}
b = \raisebox{-.45\height}{\includegraphics[scale=0.8]{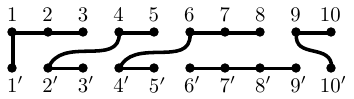}}, \, \, \,  \, \, \, 
\norm(b) = \raisebox{-.45\height}{\includegraphics[scale=0.8]{dib11.pdf}}
\raisebox{-6\height}{.}
\end{equation}
We next recall some results and conventions from \cite{Xi}.
For $ d \in  \SetPar_{2k}$ there is a canonical diagrammatic representation $ \normC(d) $ for $ d$
in which the propagating
blocks all appear with only one propagating line, which connects the leftmost points of the
corresponding top and bottom blocks. 
For example, for $ d $ as in \eqref{dib5}, we have
\begin{equation}\label{13}
d=   \raisebox{-.4\height}{\includegraphics[scale=0.8]{dib5.pdf}}, \, \, 
\normC(d) =   \raisebox{-.4\height}{\includegraphics[scale=0.8]{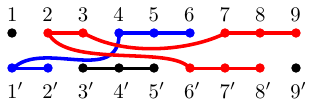}} 
\end{equation}
where we indicate with red and blue the two propagating blocks of $ \normC(d) $. 
For $l =0,1,2, \ldots, k $ we now let $  {\mathcal C}_l  $ 
be the set 
\begin{equation}\label{2.19}
  {\mathcal C}_l = \left\{ (d, S)  \,\middle\vert\, \begin{array}{l}
    d = (d_1, d_2, \ldots, d_p) \mbox{ is a set partition on } \{1,2,\ldots,k\} \mbox{ for } p \ge l \\
    S \subseteq \{ d_1, d_2, \ldots, d_p \} \mbox{ and } | S | = l
\end{array} 
  \right\}
\raisebox{-6\height}{.}
\end{equation}
Then, by \cite{Xi}, there is a bijection $ f $%  and an induced $ \CC[x] $-module isomorphism $ g $, as follows.  
\begin{equation}\label{bijection}
f:  \SetPar_{2k} {\, \, \cong \, \, }  \coprod_{l=0}^{k} { \mathcal C }_l  \times  \Si_l  \times { \mathcal C }_l \, . 
\end{equation}
For example, for $ d $ as in \eqref{13}, we have 
\begin{equation}
  f(d) = f(\normC(d) ) = 
  \big( ( d_1, {\color{red}{d_2}}, {\color{blue}{d_3}}), ( {\color{red}{d_2}}, {\color{blue}{d_3}}) \big) \times
  (1,2) \times   \big( (  {\color{blue}{d_1^{\prime}}}, d_2, {\color{red}{d_3^{\prime}}}, d_4^{\prime}), ( {\color{blue}{d_1^{\prime}}}, {\color{red}{d_3^{\prime}}}) \big)
\end{equation}
where, reading from left to right, $ d_1 = \{ 1 \} $,   $ {\color{red}{d_2}} = \{ {\color{red}{2,3,7,8,9}} \} $, corresponding to the first 
two intersection top blocks of $ d $, etc.

\medskip
We define $ \SetPar_{2k}^l   \subseteq \SetPar_{2k} $ as the set partitions whose 
diagrammatic representations
have exactly $ l $ propagating blocks and get that 
$ f $ induces a bijection $ \SetPar_{2k}^l  \cong 
{ \mathcal C }_l  \times  \Si_l  \times { \mathcal C }_l  $. 
\medskip

There are natural commuting left and right $ \Si_k $-actions on $\SetPar_{2k}^l  $ and so we also get   
left and right $ \Si_k $-actions on $ { \mathcal C }_l  \times  \Si_l  \times { \mathcal C }_l  $, via $ f$. 
These $ \Si_k $-actions on $ { \mathcal C }_l  \times  \Si_l  \times { \mathcal C }_l  $ are{\color{black}{,}} on the other
hand{\color{black}{,}} not
immediately  \lq visible\rq\, and so our first goal  
is to give another description of
$ { \mathcal C }_l  \times  \Si_l  \times { \mathcal C }_l  $ 
from which they can be read off. This
will be useful for describing a basis for $ \SpheAlg = e_k \ParAlg e_k$.

\medskip

Let $ \s,  \s_1, \T, \T_1 $ be row standard tableaux
whose shapes 
are compositions of $ k $, such that $ \s $ and $ \T $ are of length $ r $ whereas
$ \s_1  $ and $ \T_1 $ are of length $ r_1 $, where
$r $ and $ r_1  $ are both greater than or equal to $ l$. 
We then write $ (\s, \s_1) \sim_l  (\T, \T_1) $ if
$ (\s, \s_1) =  ( \rho \T, \rho_1 \T_1) $ where $ \rho $ and $ \rho_1  $
are \emph{row permutations} of $ \T $ and $ \T_1$, by which we mean 
that $ \rho $ and $  \rho_1  $ permute the rows of $ \T $ and $ \T_1 $ together with the numbers
appearing in them. 
We further require that
$ \rho $ and $  \rho_1  $
permute the first $ l $ rows of $ \T $ and $ \T_1 $ simultaneously, whereas 
they may permute the rows 
strictly below the $ l^{\color{black} \rm th} \, $row of $ \T$ and $ \T_1 $ independently.
In other words, $ \rho \in \Si_r $ and $  \rho_1 \in \Si_{r_1} $ and
$   \rho|_{ \{1, 2, \ldots, l \}} =  \rho_1|_{ \{1, 2, \ldots, l \}} $ where 
$ \rho|_{ \{1, 2, \ldots, l \}} $ and $ \rho_1|_{ \{1, 2, \ldots, l \}} $ 
denote the restrictions of $ \rho $ and $ \rho_1 $ to $ \{1, 2, \ldots, l  \} $. 
Here is an example with $  l= 3$. We indicate with red the separation of the top
$ l $ rows from the remaining lower rows of the tableaux. 
\begin{equation}\label{2.21}
  \left( \raisebox{-.5\height}{\includegraphics[scale=0.8]{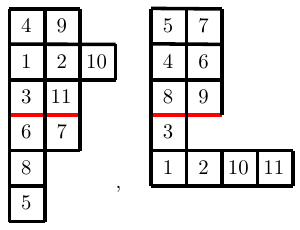}} \right) \,  \sim_3 \, 
  \left( \raisebox{-.5\height}{\includegraphics[scale=0.8]{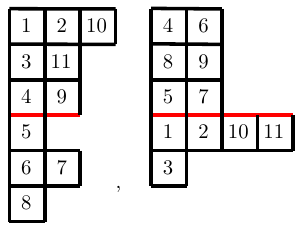}} \right) 
\raisebox{-39\height}{.}
\end{equation}
It is easy to check that $ \sim_l $ is an equivalence relation on pairs of row standard tableaux
of length greater than $ l $, 
and we define $ ( \s, \T )_{ \sim_l} $ as the equivalence class represented by $ ( \s , \T  ) $. 
Let $ i \mapsto {\rm min}_{\T}(i) $ be the function that gives the minimal (first)
number of the $ i^{\color{black} \rm th} $
%{\sout{\color{green}{th}}
row
of the row standard tableau $ \T $.
Then any class $  ( \s, \T )_{ \sim_l}  $ has a distinguished representative $  ( \s^{incr}, \T^{incr} ) $
for which $ {\rm min}_{  \s^{incr}}   $ is increasing on the restriction to $ \{1, \ldots, l \} $ 
and $ {\rm min}_{  \s^{incr}}   $ and $ {\rm min}_{  \T^{incr}}   $ are both increasing on the
restriction to $ \{l+1, l+2,  \ldots \} $. For example, in \eqref{2.21} the second pair
is the distinguished representative for its class. 

\medskip
Now $ {\rm min}_{  \T^{incr}}   $ need not be increasing on the restriction to $ \{1, \ldots, l \} $,
but there exists a row permutation $ \rho $ such that 
$ {\rm min}_{ \rho^{-1} \T^{incr}}   $ is increasing on the restriction to $ \{1, \ldots, l \} $. 
We may view $ \rho $ as an element of $ \Si_l$. For example, in \eqref{2.21} we have
$ \rho = (1,3,2) $ in permutation notation. But $ \rho $ only depends on $ (\s, \T ) $ through its 
class $ ( \s, \T )_{  \sim_l } $, 
and so we
define $\rho_{  ( \s, \T )_{ \sim_l }} =  \rho   $. 

\medskip
We next observe that any element $  d $ of $\SetPar_{2k}^l  $
gives rise to a class $ ( \s, \T  )_{\sim_l} $, by associating the intersection top blocks of $ d $ with the rows of 
$ \s $ and the intersection bottom blocks of $ d $ with the rows of $ \T $, in such a way that 
intersection top and bottom blocks that are intersections of propagating blocks for $ d $
are associated with 
the first $ l $ rows of $ \s $ and $ \T $, and with rows of the same 
row number 
if and only {\color{black}{if}} they are intersections of the same propagating block. For example, for $ d$ as in
\eqref{13} the corresponding
class is 
\begin{equation}\label{2.22}
  \left( \raisebox{-.5\height}{\includegraphics[scale=0.8]{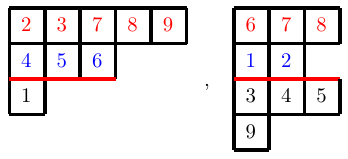}} \right)_{ \! \! \sim_2}
\raisebox{-28\height}{.} 
\end{equation}

One notes that the association just defined is a bijection between $\SetPar_{2k}^l  $ and
the set of classes $ ( \s, \T )_{ \sim_l } $. 
Note also that the $ \Si_k $-actions on
$ \SetPar_{2k}^l  $, under this bijection, correspond to the natural $ \Si_k $-actions on $  \s $ and $ \T $, 
as explained in \eqref{dib17}, although the action on $ \T $ should be chosen as a right action.

There is however also an obvious bijection between the set of classes $ ( \s, \T )_{\sim l} $ and 
$  { \mathcal C }_l  \times  \Si_l  \times { \mathcal C }_l  $. It maps $  ( \s, \T )_{\sim l} $ to
$ (d_\s, S_\s)  \times \rho_{ ( \s, \T )_{\sim l} }  \times (d_\T, S_\T) $ where $ d_\s $ is the set partition 
whose blocks are the rows of $ \s $, with $ S_\s $ being the blocks of the first $ l $ rows of $ \s $, 
and similarly for $ d_\T $ and $ S_\T $. Combining this with the bijection of the previous paragraph we have
achieved our goal of describing the $ \Si_k $-actions on
$  { \mathcal C }_l  \times  \Si_l  \times { \mathcal C }_l  $.

\medskip
We now use it to prove the following Theorem. 
\begin{theorem}\label{lemma2}
  The map $ F: \BiPar_k \rightarrow \SpheAlg$ given by $ b \mapsto e_k{\normalfont
    \norm(b)} e_k $ is injective. Moreover,
  the image of $ F $, that is
$ im  F =  \{ e_k {\normalfont \norm(b)} e_k \, | \, b \in \BiPar_k \} $, is a $\CC[x]$-basis for $ \SpheAlg $ and 
so $ \rank_{\CC[x]} \,  \SpheAlg = bp_k$. 
\end{theorem}
\begin{dem}
We first show simultaneously that $ F $ is injective and that $ im F $ is a linearly independent set. 
Let $ b \in \BiPar_k $ and consider 
$  \norm(b) $ as an element of $ \SetPar_{2k} $. Let $  ( \s, \T )_{\sim_l} $ be the class associated with
$  \norm(b) $ under the bijection explained in the paragraph before \eqref{2.22} and
let $  ( \s^{incr}, \T^{incr} ) $ be its distinguished representative, 
as defined above. Here is an example
\begin{equation}\label{2.22new}
( \s^{incr}, \T^{incr} ) =   \left\{ \raisebox{-.5\height}{\includegraphics[scale=0.8]{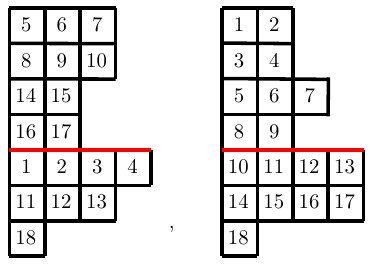}} \right\}
\raisebox{-45\height}{.} 
\end{equation}
Two properties can be observed in \eqref{2.22new} and hold for general $ ( \s^{incr}, \T^{incr} )$. 
\begin{itemize}
    \item[I.] We have $ \rho_{ (  \s, \T )_{\sim l} } = 1 $ and so $ \T^{incr} $ is the \emph{row reading tableau},
in which the numbers $ \{ 1,2,\ldots, k \} $ appear in order from left to right down the rows. Or, equivalently, 
$ {\rm min}_{  \T^{incr}}   $ is an increasing function.
\item[II.] Let $ \lambda $ be the shape of $ \s^{incr} $.  
Then $ {\rm min}_{  \s^{incr} }   $ is also increasing, but only upon restriction to
subsets $ I  $ of the row indices for $ \lambda $, for which
$\{  \lambda_i | i \in I \} $ is constant.
\end{itemize}

Using
%{\sout{\color{green}{them}}}
{\color{black}{these properties,}}
we may now argue as follows. 
Let $ \sigma, \sigma_1 \in \Si_k $ and suppose that $ \sigma \norm(b) \sigma_1 $ is of
the form $ \norm(b_1 ) $ for some $ b_1 \in \BiPar_k $. Then, passing to the pair $ ( \s^{incr}, \T^{incr} ) $
and using the properties, one sees that the only way
to obtain a{\color{black}{n}} element in normal form by acting
$ \sigma $ on $\s^{incr}$ and $ \sigma_1 $  on $\T^{incr}$ is that these two simultaneous actions only
interchange numbers appearing in the same row. 
With this, 
we deduce that 
$ b = b_1 $. 
In other words, $ \norm(b)$ is the only element
from $ \BiPar_k $ in normal form that appears in the expansion of $ e_k \norm(b) e_k $.
But this implies
that $ F $ is injective and that 
$ im F  $ is a linearly independent set, as claimed. 

\medskip
In order to prove that $ im  F $ is a spanning set, it is enough to show that  
$ e_k d e_k $ belongs to $ im  F $ for any $ d \in \SetPar_{2k} $. 
Let therefore $ ( \s, \T )_{\sim_l } $ be the class for $ d $ under
the bijection constructed before \eqref{2.22}. We first  
choose row permutations $ \rho $ and $ \rho_1 $
satisfying the conditions described in the paragraph before \eqref{2.21}, such that
$ (\rho \s, \rho_1  \T )  $ has the shape of an element corresponding to $ \norm(b) $ 
under the bijection, for some $ b \in \BiPar_k $. To be precise, 
by \eqref{lexi} this means that,
when restricted to the top $ l $
rows, the shape of $ \rho \s $ is a partition, and so are the shapes of $ \rho \s $ and $ \rho_1 \T $,
when restricted to the 
rows strictly below the $ l^{\color{black} \rm th} \,$th row, whereas $ \rho_1 \T $ is only a partition on 
the restriction to the the equally sized rows of $ \sigma \s $. Note that 
$ ( \s, \T )_{\sim l} = ( \rho \s, \rho_1 \T )_{\sim l} $. But we may at this stage choose $ \sigma, \sigma_1 \in \Si_k $
such that $ ( \sigma  \rho \s, \sigma_1 \rho_1 \T ) $ is the distinguished representative of $ \norm(b) $, for
some $ b \in \BiPar_k $ as described below \eqref{2.22new}, which shows the claim.
%{\color{black}{\sout{The Theorem is proved. }}}
\end{dem}

\section{Schur-Weyl duality for $\SpheAlg(n)$ }\label{Schur-Weyl duality}
In this section we study the specialized spherical partition algebra 
$\SpheAlg(n)$, where $ n \in \N$. Our main result
%\ppm{in this section}
is a double centralizer property
involving $\SpheAlg(n)$ and $ \Si_n $, 
both acting on the symmetric power $ S^{k} V_n $ where $ V_n$ is a $\CC $-vector space of dimension 
$ n $. It is an analogue of Schur-Weyl duality, see \cite{Schur}, \cite{Weyl}.

\medskip
Fix a basis $ \{ v_1, v_2, \ldots, v_n \} $ for $ V_n$. We consider $ V_n $ as a left $ \CC \Si_n $-module
via $ \sigma v_i = v_{ \sigma(i)} $ for $ \sigma \in \Si_n $.
Let $ V_n^{\otimes k } =\overbrace{V_n \otimes V_n \otimes \cdots \otimes V_n}^k$.
Then also $ V_n^{\otimes k } $ is a left $\CC \Si_n $-module, via the diagonal action, that is
\begin{equation} \sigma(v_{i_1} \otimes v_{i_2} \otimes \cdots \otimes  v_{i_k} ) = 
v_{ \sigma(i_1)} \otimes v_{\sigma(i_2)} \otimes \cdots \otimes  v_{\sigma(i_k)} \, \, \mbox{ for } \sigma \in \Si_n. 
\end{equation}
There is however also a natural $\CC \Si_k $-module structure on $V_n^{\otimes k } $, given by place permutation.
To distinguish it from the previous $\CC \Si_n $-module structure on $V_n^{\otimes k } $,
we choose it to be a right module structure:
%{\sout{\color{green}{in formulas:}}
\begin{equation} (v_{i_1} \otimes v_{i_2} \otimes \cdots \otimes  v_{i_k} )\sigma = 
v_{i_{\sigma(1)}} \otimes v_{i_{\sigma(2)}} \otimes \cdots \otimes  v_{i_{\sigma(k)}} \, \, \mbox{ for } \sigma \in \Si_k. 
\end{equation}
In general, the two actions commute and so $V_n^{\otimes k } $ is a
$ (\CC \Si_n, \CC \Si_k) $-bimodule. 

\medskip
We next define the {\it $k^{\color{black} \rm th} \,$symmetric power} of $ V_n$ via
\begin{equation}
S^k V_n= ( V_n^{\otimes k })e_k
\end{equation}
where $ e_k \in \CC \Si_k $ is the idempotent defined just below \eqref{dib10}.
It follows from the $ (\CC \Si_n, \CC \Si_k) $-structure on
$  V_n^{\otimes k} $ 
that $ S^k V_n $ is a left $ \CC \Si_n$-module. 

\medskip
For simplicity, we write
\begin{equation}
v_{i_1} v_{i_2} \cdots v_{i_k} = (v_{i_1} \otimes v_{i_2} \otimes  \cdots \otimes v_{i_k}) e_k 
\end{equation}
and also
\begin{equation}
  v_{i_1}^{a_1} v_{i_2}^{a_2} \cdots v_{i_p}^{a_p} =
  \bigg(    \overbrace{ v_{i_1} \otimes \cdots \otimes v_{i_1}}^{a_1 } \otimes
  \overbrace{ v_{i_2} \otimes \cdots \otimes v_{i_2}}^{a_2 }\otimes \cdots \otimes 
  \overbrace{ v_{i_p} \otimes \cdots \otimes v_{i_p}}^{a_p }
  \bigg) e_k .
\end{equation}
Then we have that
\begin{equation}\label{basis}
  \{   v_{i_1}^{m_1} v_{i_2}^{m_2} \cdots v_{i_p}^{m_p}  \, |\, 1 \le i_1 < i_2 < \ldots < i_k \le n, \,  \sum_i  m_i = k  \}
\end{equation}
is a basis for $ S^k V_n $ and so $ \dim  S^k V_n = \binom{k+n-1}{k} $.

\medskip
Our first aim is to give a decomposition of the $ \CC \Si_n$-module $ S^k V_n $ in terms of
permutation modules. 
Surprisingly, this appears to be new, and even the related $ \CC \Si_n$-decomposition 
of $ V_n^{\otimes k } $ was determined only recently in \cite{BHH}, see also 
\cite{BDM} and \cite{PMartin2}. 

\medskip
Suppose that $ \nu= (\nu_1^{a_1}, \nu_2^{a_2}, \ldots, \nu_p^{a_p} ) \in \Par_k^{\le n} $,
that is $ a_1 + a_2 + \ldots + a_p \le n $. Then, setting
$ \Phi(\nu) = \ord(a_1, a_2, \ldots, a_p, d )$ where $ d= n -( a_1 + a_2 + \ldots + a_p) $, 
we obtain a function
\begin{equation}\label{weobtainafunction}
\Phi: \Par_k^{\le n} \rightarrow \Par_n. 
\end{equation}
The following Theorem gives the promised decomposition of the $ \CC \Si_n$-module $ S^k V_n $.
\begin{theorem}\phantomsection\label{firstthm}
  \begin{description}
  \item[a)] There is an isomorphism of $ \CC \Si_n$-modules
  \begin{equation}
 S^k V_n \cong \bigoplus_{ \nu \in \Par_k^{\le n} } M( \Phi( \nu) )
  \end{equation}
where $ M( \Phi( \nu) ) $ is the permutation module as in the paragraph before \eqref{multinomial}.
\item[b)] The following multiplicity formula holds
\begin{equation}
[ S^k V_n: S(\lambda)]  = \sum_{ \nu \in \Par_k^{\le n} } {  K_{\lambda,  \Phi(\nu)} }
\end{equation}
where $ K_{\lambda,  \Phi(\nu)} $ is the Kostka number. 
  \end{description}
\end{theorem}  
\begin{dem}
In view of \eqref{multiKostka}, ${\bf b)} $ of the Theorem follows immediately from ${\bf a)} $
of the Theorem, so let us show ${\bf a)} $.

\medskip
Choose $ v=  v_{i_1}^{m_1} v_{i_2}^{m_2} \cdots v_{i_p}^{m_p} $ an element of the basis 
for $  S^k V_n$, given in \eqref{basis}, and let $ M $ be the $\CC \Si_n $-module generated by
$ v $. Note that the $i_j $'s are distinct and so there is $ \sigma \in \Si_n $ such that
\begin{equation}
\sigma(v) = v_1^{n_1} v_2^{n_2} \cdots v_p^{n_p} \, \mbox{ where } n_1  \ge n_2 \ge \ldots \ge n_p.
\end{equation}
Define now $ \nu =  (n_1, n_2, \ldots, n_p) $
and write $ \nu = (\nu_1^{a_1},  \nu_s^{a_2},  \ldots, \nu_s^{a_s} )  $
with $ \nu_1 > \nu_2 > \ldots > \nu_s $. Then one quickly checks that 
$ \sigma(v) $ generates the $\CC \Si_n $-permutation module $ M(\alpha) $
where $ \alpha= \ord(a_1, a_2, \ldots, a_s, d ) $ for $ d= n - (a_1 +a_2 + \ldots + a_s ) $,
that is $ M =  M(\alpha) $ for 
$ \alpha = \Phi(\nu) $ and $ \nu = (n_1, n_2, \ldots, n_s) $. This proves the Theorem. 
\end{dem}

\medskip
Let us illustrate the argument of the proof of the Theorem using $ k = 17$, $ n = 15 $ and
\begin{equation}
v= v_1 v_1 ( v_2 v_2 ) v_3 v_3  ( v_4 )   v_5 v_5 v_5  (v_6 v_6 ) v_7 (  v_9 ) v_{10} v_{10} v_{10}  \in  S^{17} V_{15} 
\end{equation}
where we use
{\color{black}{parentheses}} to group equal indices. Using the notation of the proof of the
Theorem, this gives   
\begin{equation}
  \sigma(v)= v_1 v_1  v_1( v_2  v_2 v_2  ) v_3 v_3  ( v_4 v_4  )   v_5 v_5   (v_6 v_6 ) v_7 (  v_8 ) v_{9}
  \end{equation}
and so $ \nu = (3,3,2,2,2,2,1,1,1) = (3^2, 2^4, 1^3) $ and $ d = 15- (2+4+3)  = 6 $,
and hence $ \alpha=\ord(2,4,3,6) =  (6,4,3,2) $. According to the Theorem we should therefore have
$  \CC \Si_{15} v = M(\alpha) $.

\medskip
On the other hand, the subgroup of $ \Si_{15} $ stabilizing $ \sigma(v) $ is
the Young subgroup
\begin{equation}
\Si_{1,2} \times \Si_{3,4,5,6} \times \Si_{7,8,9} \times \Si_{10,11,12,13,14,15} 
\end{equation}
corresponding to the multiplicities $ (2,4,3) $ of $ \nu $ and to $ d $. 
Moreover, $ \CC \Si_{15} \sigma(v)$ is spanned by the elements 
\begin{equation}\label{basiselements}
  v_{i_1} v_{i_1}  v_{i_1}( v_{i_2}  v_{i_2} v_{i_2}  ) v_{i_3} v_{i_3}  ( v_{i_4} v_{i_4}  )   v_{i_5} v_{i_5}
  (v_{i_6} v_{i_6} ) v_{i_7} (  v_{i_8} ) v_{i_9}
  \end{equation}
for distinct $ i_j  \in \{1,2,\ldots, 15 \} $. But the elements in \eqref{basiselements}
are invariant under
permutations of $ i_1 $ and $ i_2 $, permutations of $ i_3, i_4, i_5, i_6 $
and permutations of $ i_7, i_8, i_9 $
and hence there are
$ \binom{15}{ 2,  4,  3 ,6} $ of them, {\color{black}{as expected.}}
%{\color{black}{\sout{as there should be.}}}

\begin{remark}
  \normalfont
  Note that the proof of Theorem \ref{firstthm} does not use any special properties of $ \CC $ and so
  the Theorem is valid for any ground field. 
  Note also that, in view of the observation following \eqref{multiKostka}, the
  omission of $ \ord $ from the definition of $ \Phi $ in \eqref{weobtainafunction}
  does not change the validity 
  of Theorem \ref{firstthm}. 
  \end{remark}

To the best of our knowledge, the formula for the multiplicity $ [ S^k V_n: S(\lambda)]  $
in Theorem \ref{firstthm} is new, but in the theory of symmetric functions there is another
approach 
{\color{black}{to the evaluation of 
$ [ S^k V_n: S(\lambda)]  $, going back to the work of
Aitken. We make use of this alternate 
approach below.}}

\medskip
Following the notation used in \cite{Macdonald}, 
we let $ \Lambda_{\QQ}$ be the ring of symmetric functions in infinitely many variables
$ x_1, x_2, \ldots $, defined over $ \QQ$. Any basis for $ \Lambda_{\QQ}$ is indexed by $ \Par $
and one prominent basis is 
$\{ s_{\lambda} \, | \, \lambda \in \Par \} $ the basis of Schur
functions. Let $ R^k $ be the $ \QQ$-vector space 
with basis given by the irreducible characters for $ \Si_k $ and set $ R = \bigoplus_{k=0}^{\infty} R^k $
with the convention that $ R^0 = \QQ$. Let $ \ch: R \rightarrow \Lambda_{\QQ} $ be
the characteristic map. It satisfies $ \ch(\chi^{\lambda}) = s_{\lambda} $ where
$ \chi^{\lambda} $ is the character of $ S(\lambda) $. 

\medskip
Letting $ \psi^{k}_n $ be the character of the $ \Si_n $-module $ S^{k} V_n$, 
we now have that 
\begin{equation}\label{3.15}
\sum_{ k =0}^{\infty} \ch (\psi^{k}_n) t^k= \sum_{ \lambda \in \Par_n} s_{\lambda}(1,t, t^2, \ldots ) s_{\lambda}. 
\end{equation}
This is the formula showed by Aitken in 
\cite{Aitken}, see also \cite{Thibon} and exercise 7.73 in \cite{Stanley}.
For our purposes, the usefulness of it derives from the following 
expression for 
$ s_{\lambda}(1,t, t^2, \ldots ) $, see for example Corollary 7.21.3 of \cite{Stanley}. 
\begin{equation}\label{3.16}
  s_{\lambda}(1,t, t^2, \ldots ) = \dfrac{t^{b(\lambda)}}{ \prod_{ u \in \lambda} [ h(u)] }. 
\end{equation}
Here $ [ h(u)]  = 1 - t^{h(u) } $ where $ h(u ) $ is the hook length of $ u \in \lambda $, and
$ b(\lambda) = \sum_{i = 1 }^{ \ell(\lambda) } (i-1) \lambda_i $. {\color{black}{For}} example
\begin{equation}\label{3.17}
b\left( \raisebox{-.45\height}{\includegraphics[scale=0.7]{dib14.pdf}}\right) = 7. 
\end{equation}
In the notation of symmetric function theory 
the expression in \eqref{3.15} is the {\it plethystic transformation} $h_n\Bigl( \dfrac{X}{1-t} \Bigr) $
of the complete symmetric function $ h_n $ 
where $ X= x_1 + x_2 +\ldots $, 
see for example Proposition 3.3.1 of the survey paper \cite{Ha}. Since $ h_n = s_n $, it is also equal to 
$s_n\Bigl( \dfrac{X}{1-t} \Bigr) $.
Recall that plethystic transformation plays an important role 
in the theory of integrality and positivity of Macdonald polynomials.
Indeed, these integrality and positivity properties only hold for the
plethystically transformed Macdonald polynomials, not for the original Macdonald polynomials.

\medskip

Combining the two formulas \eqref{3.15} and \eqref{3.16}, one gets an expression for the
multiplicity $ [ S^k V_n: S(\lambda)]  $ by
{\color{black}{taking}} the coefficient of $ t^k $ in the power
series expansion of \eqref{3.16}.
This is less concrete than our closed formula in Theorem
\ref{firstthm}, but, as we shall now see, it allows us to determine exactly when $  [ S^k V_n: S(\lambda)]  \neq 0 $. 

\begin{lemma}\label{exactlywhen}
  In the above setting we have that $  [ S^k V_n: S(\lambda)]  \neq 0 $
if and only if $ k \ge b(\lambda) $.   
\end{lemma}
\begin{dem}
  If $ k < b(\lambda) $, it follows immediately from \eqref{3.15} and \eqref{3.16} that
$  [ S^k V_n: S(\lambda)]  =0 $. Conversely, if $ k \ge b(\lambda) $ it follows from 
  \eqref{3.15} and \eqref{3.16} that $  [ S^k V_n: S(\lambda)]  \neq 0 $ since any partition
  $ \lambda \in \Par_n $ has at least one node $ u $ of hook length 1 which gives a 
  contribution $ \frac{t^{b(\lambda)}}{ [ h(u)] } = t^{b(\lambda)} (1 + t + t^2 + \ldots ) $ to
\eqref{3.16} that cannot be cancelled out. 
 \end{dem}

\medskip
In view of the Lemma we now define
\begin{equation}\label{defineLambaSph}
  \ParSph = \{ \lambda \in \Par_n \,| \,  b(\lambda) \le  k    \} .
\end{equation}
For $ k $ big enough, we have $   \ParSph = \Par_n$. The next Lemma makes
this statement precise. 
\begin{lemma}
We have $\ParSph = \Par_n$ if and only if $ \, \dfrac{n(n-1)}{2}   \le k $. 
\end{lemma}  
 \begin{dem}
   For $ \lambda \in \Par_n $ we interpret $ b(\lambda ) $ as the sum of all the entries 
   of the semistandard $ \lambda$-tableau $ \T$ on $ \{0,1,2,\ldots, n-1 \}$,
   obtained by inserting $ 0 $ in all the nodes of the
   first row of $ \lambda$, $1$ in all the nodes of the second row of $ \lambda$, and so on.
   For example, for $ \lambda $ as in \eqref{3.17} we have that
\begin{equation}
  \T=  \raisebox{-.45\height}{\includegraphics[scale=0.7]{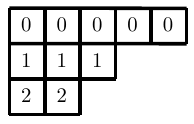}}
\raisebox{-15\height}{.} 
\end{equation}
In view of this interpretation, it is clear that for $ \lambda $ running over $\Par_n $, the
maximal value of $ b(\lambda) $ is obtained for 
the one column partition $ \lambda = (1^n) $. But for this $ \lambda$ we have
$ b(\lambda) = \dfrac{n(n-1)}{2}$, {\color{black}{which proves the desired result}}. 
 \end{dem}

\medskip
We now turn to our Schur-Weyl duality statement. It was shown in \cite{Jones} and \cite{PMartin} that $ V_n^{\otimes k } $ 
is a right module for $  \ParAlg(n)$,  
with action commuting with
the left $ \CC \Si_n $-action on $ V_n^{\otimes k } $ and
so $ V_n^{\otimes k } $ is a $ (\CC \Si_n, \CC \Si_k) $-bimodule. 
We do not need 
the actual formulas that define this $  \ParAlg(n)$-action, only the facts 
that the induced algebra homomorphism 
\begin{equation}\label{tobeprecise}
  \Upsilon:  \ParAlg(n) \twoheadrightarrow \End_{ \CC \Si_n}( V_n^{ \otimes k } ), \,
\Upsilon(p)(v) = vp,  \mbox{ where } p \in \ParAlg(n), v\in  V_n^{ \otimes k }
\end{equation}
is surjective and is an isomorphism if $ n \ge 2k $. 
The $  \ParAlg(n)$-action on $ V_n^{\otimes k } $ 
induces an
$ \SpheAlg(n) =e_k  \ParAlg(n)  e_k $-action on $ S^k V_n = (V_n^{\otimes n }) e_k  $, and hence an 
algebra homomorphism 
\begin{equation}\label{onemayexpect}
  \Upsilon_{sph}:  \SpheAlg(n) \rightarrow \End_{ \CC \Si_n}( S^k V_n ), \, \,
\Upsilon_{sph}(e_k p e_k)(v)   = v   e_k p e_k
 \mbox{ where } e_k p e_k  \in \SpheAlg(n), v\in   S^k V_n. 
\end{equation}
On the other hand, there is also an algebra homomorphism
\begin{equation}\label{secondstatementSchurWeyl}
  \Xi:  \CC \Si_n  \twoheadrightarrow \End_{ \ParAlg(n) }( V_n^{ \otimes k } ), \,
\Xi(x) = x v,  \mbox{ where } x \in \Si_n, v\in  V_n^{ \otimes k }
\end{equation}
which is surjective, as follows from the surjectivity of $ \Upsilon $ 
and 
Burnside's density {\color{black}{theorem}}, see for example
\cite{Lang} or Theorem 5.4 in \cite{HR}, and Maschke's Theorem for $ \CC \Si_n $.
It induces a homomorphism 
\begin{equation}\label{secondstatementSchurWeylA}
  \Xi_{sph}:  \CC \Si_n  \rightarrow \End_{ \SpheAlg(n) }( S^k V_n  ), \,
\Xi(x) = x v,  \mbox{ where } x \in \Si_n, v\in  S^k V_n.
\end{equation}

The algebra surjections in \eqref{tobeprecise} and \eqref{secondstatementSchurWeyl}
express the statement that the commutating actions of $ \ParAlg(n) $ and $  \CC \Si_n $ on $ V^{\otimes k}_n $
centralise each other, and therefore are in 
{\it Schur-Weyl duality}
on $ V_n^{ \otimes k } $. 

\medskip 
Note that in the statistical mechanical model underpinning the partition algebra $ \ParAlg(n) $,
that is the Potts model, the $ \ParAlg(n) $-module $ V^{\otimes k}_n $ is the {\it $n$-state Potts representation},
see \cite[\S8.2]{PMartin}. In this setting, the commuting action of $ \Si_n $ is {\it the Potts symmetry}. 

\medskip

In view of \eqref{tobeprecise} and \eqref{secondstatementSchurWeyl}, 
one may now hope that $ \SpheAlg(n) $ and $  \CC \Si_n $ are in Schur-Weyl duality
on $ S^k V_n  $, via the maps 
$ \Upsilon_{sph} $ and $   \Xi_{sph} $ given in \eqref{onemayexpect} and
\eqref{secondstatementSchurWeylA}. 
Our next result is that this indeed is the case.
\begin{theorem}\phantomsection\label{teorem 6}
  \begin{description}
  \item[a)] 
    The algebra homomorphism $   \Upsilon_{sph} $ is surjective for all $k, n $ and it 
  is an isomorphism if $ n \ge 2k$. 
\item[b)]    The algebra homomorphism $   \Xi_{sph} $ is surjective for all $k, n $. 
\end{description}  
  \end{theorem}
  \begin{dem}
  Let us first show that $ \Upsilon_{sph} $ is surjective. Suppose that $ f \in \End_{ \CC \Si_n}( S^k V_n ) $.
Since $ e_k $ is an idempotent in $ \ParAlg(n) $ we have that $ S^k V_n $ is a $ \CC \Si_n $-summand of
$ V_n^{ \otimes k } $, that is $ V_n^{ \otimes k } \cong  S^k V_n \oplus M $ where $ M $ is 
the $ \CC \Si_n $-module $ M = V_n^{ \otimes k }(1-e_k) $. Hence $ f$ can be extended 
to an endomorphism $ f_{ext} \in  \End_{ \CC \Si_n}(  V_n^{ \otimes n} ) $, via $ f_{ext} = (f, 0 ) $ along
this decomposition.
But then, by \eqref{tobeprecise}, there is $ p \in \ParAlg(n) $ such that
$ f_{ext} = \Upsilon(p) $ from which we deduce that $ f = \Upsilon(e_k p e_k) $.
This shows 
surjectivity of $ \Upsilon_{sph}$.

\medskip
We next assume $ n \ge 2k $ and calculate $  \dim\End_{ \CC \Si_n}( S^k V_n ) $.
Using the basis in \eqref{basis}, an 
element $ f$ of $   \End_{ \CC} ( S^k V_n ) $ can be described as a
$ \binom{k+n-1}{k} \times \binom{k+n-1}{k} $ matrix
  $ A = \left(a_{i_1, i_2, \ldots, i_k }^{j_1, j_2 \ldots, j_k } \right)  $ for increasing sequences 
$ i_1 \le  i_2 \le \ldots \le i_k \le n $ and $ j_1 \le  j_2 \le \ldots \le j_k \le n$. 
The condition that $ f $ is $ \CC \Si_n $-linear corresponds to requiring additionally that 
\begin{equation}\label{condition}
\left(a_{i_1, i_2, \ldots, i_k }^{j_1, j_2 \ldots, j_k }\right)
  = \left(a_{\Ord(\sigma(i_1), \sigma(i_2), \ldots, \sigma(i_k)) }^{\Ord(\sigma(j_1), \sigma(j_2) \ldots, \sigma(j_k ))}\right)
\mbox{   for all  }
 \sigma \in \Si_n
\end{equation}
where $ \Ord$ is the function that reorders the elements of a sequence to produce a weakly increasing sequence. 
For weakly increasing sequences $ (r_1, r_2, \ldots r_k) $ and $ (s_1, s_2, \ldots s_k) $ over $ \{1, 2, \ldots, n \} $\
we define the matrix
$
A_{r_1 , r_2 , \ldots,  r_k}^{  s_1 , s_2 , \ldots,  s_k} =
\left(a_{i_1, i_2, \ldots, i_k }^{j_1, j_2, \ldots, j_k }\right)
$ via
\begin{equation}
  a_{i_1, i_2, \ldots, i_k }^{j_1, j_2 \ldots, j_k } =
  \begin{dcases}
    1 & \mbox{if there exists }  \sigma \in \Si_n \mbox{ such that:}\begin{array}{l} (i_1, i_2, \ldots, i_k)=  \Ord(\sigma(r_1), \sigma(r_2), \ldots, \sigma(r_k)) \mbox{ and } 
\\  (j_1, j_2, \ldots, j_k)=  \Ord(\sigma(s_1), \sigma(s_2), \ldots, \sigma(s_k))
\end{array}
    \\
    0 & \mbox{otherwise.}
  \end{dcases}
\end{equation}
Then, by \eqref{condition}, 
the distinct matrices $ A_{r_1 , r_2 , \ldots,  r_k}^{  s_1 , s_2 , \ldots,  s_k} $ form
a basis for $\End_{ \CC \Si_n}( S^k V_n ) $. We arrange pairs of weakly increasing sequences
  $ (r_1, r_2, \ldots r_k) $ and $ (s_1, s_2, \ldots s_k) $ over $ \{1, 2, \ldots, n \} $ in the form
$\small \begin{pmatrix*}[r]
    s_1, s_2 \ldots s_k \\  r_1, r_2 \ldots r_k  \end{pmatrix*} $ and then get 
an $ \Si_n $-action on them via $ \small \sigma \begin{pmatrix*}[r]
  s_1, s_2 \ldots s_k \\  r_1, r_2 \ldots r_k  \end{pmatrix*} =  \begin{pmatrix*}[r]
 \Ord( \sigma(s_1), \sigma(s_2) \ldots \sigma(s_k)) \\ \Ord( \sigma(r_1), \sigma(r_2) \ldots \sigma(r_k))  \end{pmatrix*} $. Then each matrix 
$ A_{r_1 , r_2 , \ldots,  r_k}^{  s_1 , s_2 , \ldots,  s_k} $ only depends on the $ \Si_n $-orbit of
$\small  \begin{pmatrix*}[r]
  s_1 s_2 \ldots s_k \\  r_1 r_2 \ldots r_k  \end{pmatrix*} $ 
and these orbits 
are in bijection with 
bipartite partitions in $\BiPar_k $ by letting equal numbers belong to the same
part. For example, for $ k = 16, n = 5 $ we have that 
\begin{equation}
\begin{pmatrix*}[r]
   1111 22 33 444 5555 \\    11 222 333333 5555  \end{pmatrix*} \mapsto
  \raisebox{-.35\height}{\includegraphics[scale=0.9]{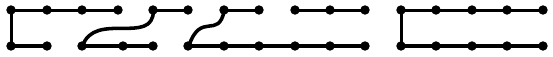}}
\raisebox{-5\height}{.} 
\end{equation}
Moreover, by the assumption $ n \ge 2k $, each $ b \in   \BiPar_k $ arises this way
from such an $ \Si_n $-orbit, and hence $ \dim \End_{ \CC \Si_n}( S^k V_n ) = bp_k.$
Combining this with Theorem \ref{lemma2} we get that $ \dim \SpheAlg(n) = \dim \End_{ \CC \Si_n}( S^k V_n ) $
  and so $ \Upsilon_{sph} $ is an isomorphism if $ n \ge 2k $. This proves ${\bf a)} $ of the Theorem,
  and ${\bf b)} $ follows from Burnside's density {\color{black}{theorem}}, once again, 
and Maschke's Theorem for $ \CC \Si_n $. 
\end{dem}  

  \medskip
  Define now $  Z_{sph}^{ k, n} $ as the image of $ \Upsilon_{sph} $, that is 
  as the centralizer algebra $  Z_{sph}^{ k, n} = \End_{ \CC \Si_n}( S^k V_n )$.
By joining the results of this section we get the following Theorem.
\begin{theorem}\phantomsection\label{joining the results} 
  \begin{description}
  \item[a)] The irreducible $   Z_{sph}^{ k, n} $-modules are indexed by
    $ \ParSph $, see \eqref{defineLambaSph}. 
  \item[b)]
For $ \lambda \in \ParSph $, let 
$ G_k(\lambda)  $ be the irreducible $Z_{sph}^{ k, n}  $-module given in $ \bf a)$. 
Then there is an isomorphism of $ (\CC \Si_n, \SpheAlg(n)) $-bimodules
    \begin{equation}
S^k V_n \cong \bigoplus_{ \lambda \in  \ParSph} S(\lambda) \otimes G_k(\lambda) 
    \end{equation}
    where $ G_k(\lambda) $ is viewed as an $\SpheAlg(n)$-module via inflation along 
    $ \SpheAlg(n) \rightarrow Z_{sph}^{ k, n}$. 
  \item[c)]
    For $ \lambda \in \ParSph $, we have $ \dim G_k(\lambda) =
    \sum_{ \nu \in \Par_k^{\le n} } {  K_{\lambda,  \Phi(\nu)} }$.
  \item[d)] $ Z_{sph}^{ k, n} $ is a semisimple algebra and 
    $ \dim Z_{sph}^{ k, n} = 
  \sum_{ \lambda \in  \ParSph } { (\dim G_k(\lambda) )^2 }$. 
  \end{description}  
  \end{theorem}

\begin{remark}\label{remark2}
  \normalfont
The Theorem should be contrasted with Theorem 3.22 in \cite{HR}, describing
the decomposition of $ V_n^{\otimes k } $ as a $ (\CC \Si_n, \ParAlg(n)) $-bimodule. In that \lq classical\rq\
setting the
  role played by our $ \ParSph$ is replaced by $ \ParPar $ defined as 
\begin{equation}
  \ParPar = \{ \lambda=(\lambda_1, \lambda_2, \ldots, \lambda_l )  \in \Par_n \, | \, \lambda_2 + \lambda_3 + \ldots + \lambda_l \le k \}.
\end{equation}
Note however that the proofs from the classical situation do not carry over to our setting.
\end{remark}

\begin{figure}[h]
\raisebox{-.35\height}{\includegraphics[scale=0.7]{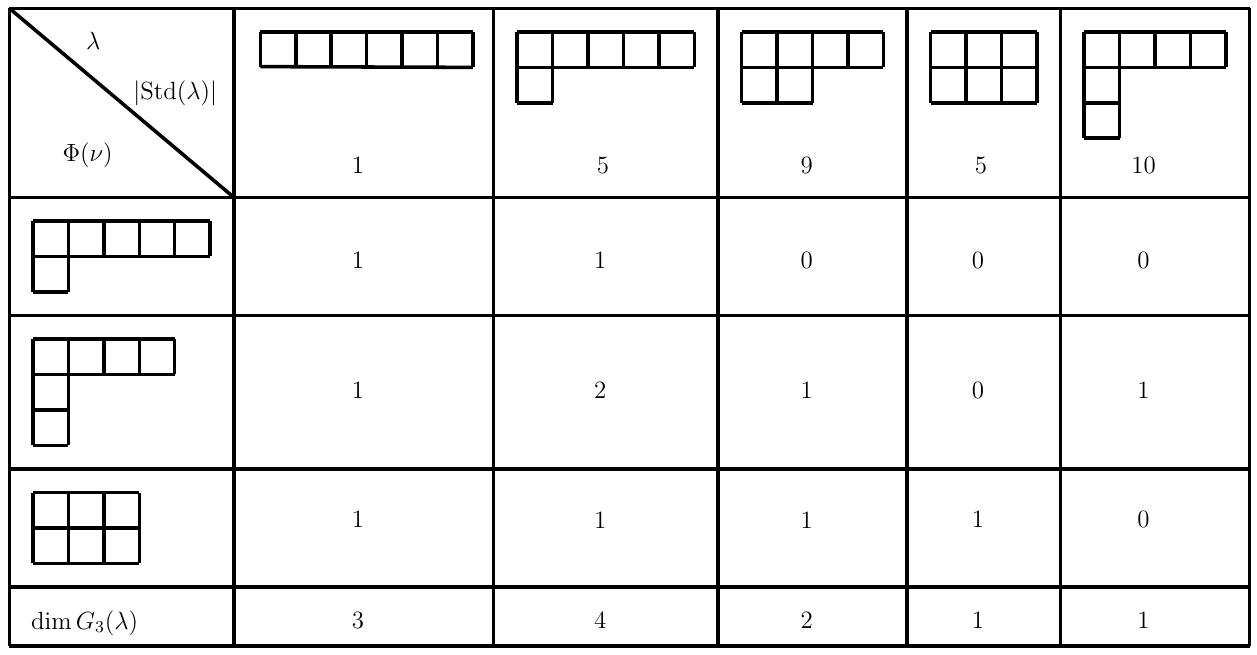}}\, \, 
\raisebox{-70\height}{.} 
\centering
\caption{Example using $n=6, k=3$.}
\label{table1}
\end{figure}

Let us illustrate ${ \bf d)} $ of Theorem \ref{joining the results}, using $ n = 6 $ and $ k = 3$. In
that case $ n \ge 2k $ and so by  
Theorem \ref{lemma2} and Theorem \ref{teorem 6} we have 
$ \dim Z_{sph}^{ 3, 6} = \dim \SpheAlgThree = 31 $. On the other hand, from
\eqref{defineLambaSph} we get 
$  \ParSphThreeSix = \{ (6), (5,1), (4,2), (3,3), (4,1,1) \} $ and since
$ \Par_3^{\le 6} = \Par_3 = \{ (3), (2,1), (1^3) \} $ we have via the definition
of $ \Phi $ in \eqref{weobtainafunction} that 
$ \{ \Phi(\nu ) \, | \, \nu \in \Par_3^{\le 6} \} = \{ (5,1), (4,1,1), (3,3) \} $. 
The table in {\color{black}{Figure \ref{table1}}} gives the Kostka numbers $ K_{\lambda, \Phi(\nu) } $ and hence 
$ \dim G_3(\lambda) $ for $ \lambda \in  \Par_{sph}^{3,6}$, via ${ \bf c)} $ of the Theorem. 

\medskip
Summing the squares of the numbers of the last row of the table we get $ 3^2+ 4^2 + 2^2 +1^2 + 1^2 = 31$,
{\color{black}{as expected}}.

\medskip
Similarly, we can use the table to illustrate ${ \bf b)} $ of Theorem
{\color{black}{\ref{joining the results}}}, at least at dimension level.
Indeed, summing the products of the numbers of the first and the last row we
get $ 1 \times 3 + 5 \times 4 + 9 \times 2 + 5 \times 1 +  10\times 1  = 56 =
\dim S^{ \color{black}{3}} V_{\color{black}{6}} $.

\begin{remark}
  \normalfont
  As already mentioned in the introduction, A. Wilson has shown that $ \SpheAlg $ coincides with the
  multiset partition algebra $ \mathcal{MP}_k(x) $ that was introduced in \cite{NaPaulSriva}. The definition
  of $ \mathcal{MP}_k(x) $ is quite different from the definition of $ \SpheAlg $, but in Lemma 5.12 of 
  \cite{NaPaulSriva} the authors prove that $ \mathcal{MP}_k(x) $ arises from $ \ParAlg $ via idempotent truncation
  with respect to a certain idempotent $  e_k^{\prime}$, defined in terms of the {\it orbit basis} for $ \ParAlg $. 
Wilson shows that 
the two idempotents $  e_k^{\prime}$ and $     e_k$ in fact coincide. 
\end{remark}

{\color{black}{
\begin{example}\label{example 1}
  \normalfont
Suppose that $n \ge 2k $. Then 
by Remark \ref{remark2} the partitions
$ (n-k, k) $ and $ (n-k, 1^k)  $ both belong to $  \ParPar $. Moreover, 
by Lemma \ref{exactlywhen}, we also have that $  (n-k, k)  $ belongs to 
$  \ParSph $ but $ (n-k, 1^k)  $ does not. 
\end{example}

\begin{remark}
  \normalfont In analogy with $\SpheAlg $, it would seem natural also to introduce an
              {\it antispherical partition algebra $ \AntiSpheAlg$} via $ \AntiSpheAlg  = f_n \ParAlg f_n  $, 
              where $ f_n = \iota_k\left ( \frac{1}{k! } \sum_{ \sigma \in \Si_k }  \Sign (\sigma)\sigma \right) $ and 
              where $ \Sign(\sigma) $ is the usual sign of $ \sigma \in \Si_k $. 
On the other hand, for any transposition $ \sigma \in \Si_k $ we have
that $ \sigma f_n = f_n \sigma = -f_n $ and so
$ \AntiSpheAlg $ 
is a small algebra, since in fact $\rank_{\CC[x]} \, \AntiSpheAlg = 2 $
for $ k \ge 2$.

\medskip
Even so, if $ n \ge 2k $, one could   
still develop analogues for $\AntiSpheAlg$ of our results for $\SpheAlg$, by replacing 
$ S^k V_n $ with the exterior power module
$  \bigwedge^{\! k} V_n = (V^{\otimes n}) f_n $.
Then $\AntiSpheAlg$ is in Schur-Weyl duality with $  \CC \Si_n $ on $ \bigwedge^{\! k} V_n $ and
we have $ \CC \Si_n $-module isomorphisms 
\begin{equation} \label{multione}
\textstyle   \bigwedge^{\!  k}  
V_n  \cong {\rm Ind}_{  \Si_{n-k} \times \Si_{k }}^{\Si_n}  \left( S(n-k ) \otimes S(1^k ) \right) 
\cong S(n-k, 1^k ) \oplus S(n-k+1, 1^{k-1} )
\end{equation}
where the last isomorphism follows from the Littlewood-Richardson rule. 
The two Specht modules appear with multiplicity one in \eqref{multione}, and so we deduce 
that $\AntiSpheAlg$ has two simple modules, each of dimension one.
This is in accordance with $\rank_{\CC[x]} \, \AntiSpheAlg = 2 $.

\medskip
We shall not consider $\AntiSpheAlg$ further in the paper.

\end{remark}
}}

\section{Cellularity of $\SpheAlg(t)$ }\label{sectioncellularity}
In this section we initiate the study of the representation theory of $\SpheAlg(t)$,
for arbitrary $ t \in \CC$.

\medskip
It was shown in \cite{PMartin1} that 
$ \ParAlg(t) $ is semisimple if and only if $ t \notin \{ 0, 1, 2, \ldots, 2k -2 \} $. 
This gives us immediately the following Theorem.
\begin{theorem}\label{semisimple}
  Suppose that $ t \notin \{ 0, 1, 2, \ldots, 2k -2 \} $. Then $\SpheAlg(t)$
  is a semisimple algebra. 
\end{theorem}
\begin{dem}
  Let $ \mathcal{J}_k $ and $ \mathcal{SJ}_k $ be the Jacobson radicals for $ \ParAlg(t) $ and
  $\SpheAlg(t)$, respectively. Then, by definition, $ a \in \mathcal{J}_k $ if and only if $ a L = 0 $ for all
  irreducible $ \ParAlg(t) $-modules, and similarly for $ \mathcal{SJ}_k $.

  \medskip Since $ t \notin \{ 0, 1, 2, \ldots, 2k -2 \} $ we have that $ \ParAlg(t) $ is semisimple, which by definition means that
  $ \mathcal{J}_k =0 $. On the other hand, it is known that the irreducible $\SpheAlg(t)$-modules are the nonzero
  $  e_k L $'s for $ L $ running over irreducible $ \ParAlg(t) $-modules, see (iv) of Theorem (4) of {\bf A1} of the appendix to \cite{Donkin}. 
  Suppose now that $ e_k a e_k \in \mathcal{SJ}_k$. Then $ e_k a e_k (e_k L)=0 $ and hence
  $ e_k a e_k  L=0 $ for all irreducible
  $ \ParAlg(t) $-modules $L$. But this means that $ e_k a e_k \in \mathcal{J}_k $ and so $ e_k a e_k =0 $, as claimed.
\end{dem}  

\medskip

In general, even when $  \ParAlg(t) $ is not semisimple, it is always a {\it cellular algebra} in the sense of \cite{GL},
as was shown in \cite{DW} and \cite{Xi}, 
and so $\SpheAlg(t)$
becomes a cellular algebra as well, since
it is an idempotent truncation of $ \ParAlg(t) $. 

\medskip
Let us give the details of this statement, starting with the definition of a cellular algebra from \cite{GL}. 
\begin{definition}{\label{cellular}} Suppose that $\mathcal{A}$ is a
  $\Bbbk $-algebra over the domain $ \Bbbk $. Suppose moreover that $(\PP,\leq)$ is a poset
  such that for each $\lambda\in \PP$ there is a
finite set $T(\lambda)$ and elements $C_{\s\T}\in \mathcal{A}$ such that
\begin{equation}
  \mathcal{C}=\left\{C_{\s\T}\mid \lambda\in \PP \mbox{ and } \s,\T\in T(\lambda)\right\}
\end{equation}  
is a $\Bbbk$-basis for $\mathcal{A} $. Then the pair $(\mathcal{C},\PP)$ is called a 
\textit{cellular basis} for $\mathcal{A}$ if
\begin{enumerate}\renewcommand{\labelenumi}{\textbf{(\roman{enumi})}}
\item The $\Bbbk$-linear map $*:\mathcal{A}\to \mathcal{A}$ determined by
$(\mathop{C_{\s\T}})^*=C_{\T\s}$ for all
$\lambda\in\PP$ and $\s,\T\in T(\lambda)$ is an algebra
anti-automorphism of $\mathcal A$.

\item For any $\lambda\in \PP,\; \T\in T(\lambda)$ and $a\in \mathcal{A}$
there exist elements $ r_{a \s \U} \in \Bbbk$ such that for all $\s\in T(\lambda)$
\begin{equation}\label{multicell}
  a C_{\s\T}  \equiv \sum_{\U\in T(\lambda)} r_{a \s \U}C_{\U\T} \mod{\mathcal{A}^{> \lambda}}
  \end{equation}
where $\mathcal{A}^{    >\lambda} $ is the free $\Bbbk$-submodule of $\mathcal{A}$, given by 
$\left\{C_{\U\V}\mid \mu\in \PP,\mu>\lambda \mbox{ and }
\U,\V\in T(\mu)\right\}$.
\end{enumerate}
If $\mathcal A$ has a cellular basis we say that it is a \textit{cellular
algebra} with \textit{cell datum} $(\PP,T,\mathcal{C})$.
\end{definition}

Suppose that $\mathcal A$ is a a cellular algebra with cell datum $(\PP,T,\mathcal{C})$.
With each $ \s \in T(\lambda) $ we associate a symbol $ C_{\s} $ and next define
$ \Delta(\lambda) $ as the free $  \Bbbk $-module with basis $ \{ C_{\s} \,| \, \s \in T(\lambda) \} $. 
Then $ \Delta(\lambda) $ becomes a left $ \mathcal A$-module, called the {\it cell module}, via
\begin{equation}
a C_{\s} = \sum_{\U\in T(\lambda)} r_{a \s \U}C_{\U}
\end{equation}  
where $ r_{a \s \U} $ is as in \eqref{multicell}. We shall call $ \{ C_{\s} \, | \, \s \in T(\lambda) \} $
the {\it cellular basis} for $ \Delta(\lambda) $.

\medskip
We now state the cell datum for $ \ParAlg(t)$, using a small variation of the
constructions given in \cite{DW} and \cite{Xi}. 
For $ \Lambda $ we use
\begin{equation}\label{poserpar}
\Lambda^k  = \bigcup_{l=0}^{k} \Par_l. 
\end{equation}
For the order relation
$ \unlhd $ on $ \Lambda^k $ we use the usual dominance order on each $ \Par_l $, and
extend it to all of $ \Lambda^k $ via $ \lambda \lhd \mu $ if $ \lambda \in \Par_l $
and $ \mu \in \Par_{\overline{l}} $ where $ l >\overline{l} $.
Suppose that $ \lambda \in \Par_l \subseteq \Lambda^k$.
Then for $ T(\lambda) $ we use $ T_k(\lambda ) = \std(\lambda) \times {\mathcal C}_l $ where $ {\mathcal C}_l $
is as in \eqref{2.19}. Thus, the elements of $ T_k(\lambda) $ are of the form  $ \cc = ( \s, c , S) $
where
$ \s \in \std(\lambda) $ for $ \lambda \in \Par_l$, and $ c $ is a set partition on $ \{ 1,2, \ldots, k \} $
with $ S$ being a subset of the blocks of $ c$, such that $ | S | = l $. 

Finally, in order to give the cellular basis itself,
we need to recall Murphy's {\it standard basis} for $ \CC \Si_l $.
For $ \lambda \in \Par_l $, we denote by $ \T^{\lambda} $ the {\it row reading} tableau that
was already used in the proof of Theorem
\ref{lemma2}. In $ \T^{\lambda} $, the numbers $ \{1, 2, 3, \ldots, l \} $
are filled in increasingly along the rows of $ \lambda$ and down the columns, for example
for $ \lambda = (5,3,2) $ we have 

\begin{equation}
\T^{\lambda} =  \raisebox{-.45\height}{\includegraphics[scale=0.7]{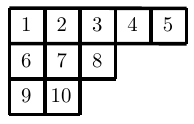}}
\raisebox{-15\height}{.} 
\end{equation}

Let $ \Si_{\lambda } \le \Si_l $ be the Young subgroup for $ \lambda $, 
that is the row stabilizer 
of $ \T^{\lambda} $, and define $ x_{\lambda \lambda } \in \CC \Si_l $ via 
$ x_{\lambda \lambda } = \sum_{ w \in \Si_{\lambda}} w $.
For $ \s \in \tab(\lambda) $, let $ d(\s) \in \Si_l $ be defined by the condition that $ d(\s) \T^{\lambda} = \s $, 
and for $ \s, \T \in \tab(\lambda) $ let $ x_{ \s \T} = d(\s) x_{\lambda \lambda } d(\T)^{-1} $. 
Then it was proved in \cite{Mathas} and \cite{Murphy}
that the set $ \{ x_{ \s \T} \, | \, \s, \T \in \std(\lambda), \lambda \in \Par_l \} $
is a cellular basis for $ \CC \Si_l $: Murphy's standard basis. 
(In fact, in \cite{Mathas} and \cite{Murphy} the authors work in the more general setting of
Hecke algebras of type $ A_{l-1}$).

\medskip
Let $ \II_l^{ \rhd \lambda} = \spa \{ x_{\s \T} \, | \,
\s, \T \in \std(\mu), \mu \rhd \lambda \} \subseteq \CC \Si_l $ be the {\it cell ideal}
in $ \CC \Si_l $ 
corresponding to $ \lambda $
and let $ x_{\s } = x_{ \s \T^{\lambda}} \, \, {\rm mod} \, \,  \II_l^{\rhd \lambda} \subseteq   \CC \Si_l/\II_l^{\rhd \lambda}
$. When $  \T^{\lambda} $ appears as a subscript, we sometimes
write $ \lambda $ instead of $ \T^{\lambda} $, 
for example $ x_{\s \lambda} =  x_{ \s \T^{\lambda} }  $ and $ x_{\lambda} =  x_{  \T^{\lambda} }  $.
Then the Specht module $ S(\lambda) $ for $ {\mathbb C}\Si_l $ is the submodule of $  {\mathbb C}\Si_l/ \II_l^{ \rhd \lambda} $ generated by $ x_{\lambda} $. It is the cell module associated with Murphy's standard basis and 
$ \{ x_{\s } \, | \, \s \in \std(\lambda) \} $
is a cellular basis for $ S(\lambda) $. 

\medskip
Returning to $  \ParAlg(t) $ we finally obtain 
its cellular basis. 
For $ \cc = (\s, c , S ) $ and $  \dd = (\T, d, T) $ in $  T_k(\lambda)$ we define 
$ C_{\cc \dd} \in  \ParAlg(t)  $ via 
\begin{equation} C_{\cc \dd}  = g \bigl ( (c, S) \otimes  x_{\s \T} \otimes (d, T) \bigr )
\end{equation}
where $ g $ is the isomorphism induced by $ f^{-1} $ for $ f $ as in \eqref{bijection}. 
Then $ \{ C_{\cc \dd} \, | \,  \cc, \dd \in T_k(\lambda) \, \,   \mbox{for} \, \, \lambda \in \Lambda^k \} $ is
the cellular basis for $ \ParAlg(t) $. 
A typical basis element $ C_{\cc \dd} $ has the diagrammatic form 

\begin{equation}\label{before dd}
C_{\cc \dd}  = \raisebox{-.47\height}{\includegraphics[scale=0.8]{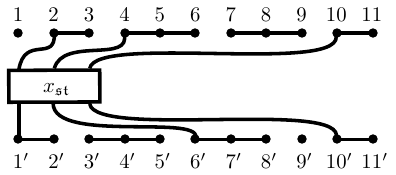}}
\raisebox{-20\height}{.} 
\end{equation}

For $ \lambda \in \Lambda^k $, we now give a description 
of the cell module $ \Delta_k(\lambda) $ for $ \ParAlg(t) $. For 
$ \lambda \in \Par_l \in \Lambda^k $ we  
let $ \dd_{\lambda} \in T(\lambda) $ be the element defined via
$ \dd_{\lambda} = ( \T^{\lambda}, d, T ) $ where $ T =  \{ \{1\}, \{2\}, \ldots, \{ l\} \}  $
and $ d =  \{ \{1\}, \{2\}, \ldots, \{ l\},   \{ l+1, l+2, \ldots, k \} \}  $. 
For $  \cc = (\s, c, S) \in T(\lambda) $ we set
\begin{equation}\label{6.8}
C_{\cc} = C_{\cc \dd_{\lambda}} \mod \ParAlg^{  \rhd \lambda}(t)
\end{equation}
where 
$ \ParAlg^{\rhd \lambda}(t) = {\rm span}\{ C_{ \cc \dd} \, | \, \cc, \dd \in T(\mu),  \mu \rhd \lambda \} $
and have then $ \Delta_k(\lambda) =   {\rm span} \{ C_{\cc} \, | \, \cc \in  T_k(\lambda) \} $.
Then, by definition, $ \Delta_k(\lambda) $ is the submodule of $ \ParAlg(t)/ \ParAlg^{\rhd \lambda}(t) $ generated by 
$ \{ C_{\cc} \, | \, \cc \in  T_k(\lambda) \} $.
We represent a typical basis element $  C_{\cc} $ for $ \Delta_k(\lambda) $ as a {\it half diagram} as follows 
\begin{equation}\label{6.9}
  C_{\cc} = \raisebox{-.24\height}{\includegraphics[scale=0.8]{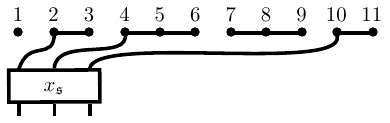}}
\end{equation}
thus leaving out $ \dd_{\lambda} $ from the diagram. 
The action of $ a \in \ParAlg(t) $ on $ C_{\cc} \in \Delta_k (\lambda) $,
that is $ a  C_{\cc} \in \Delta_k (\lambda)$, is given by concatenation
with $ a $ on top of $ C_{\cc} $, followed by the elimination of internal blocks as in $ \ParAlg(t)$, 
and of terms involving $ \{ C_{ \dd} \, | \, \dd \notin T_k(\lambda)  \} $ that are set equal to $ 0 $. 

\medskip
By construction we have 
\begin{equation}\label{bytheconstruction}
\dim \Delta_k(\lambda) = | T_k(\lambda) | = | \std(\lambda) | | C_l | 
\end{equation}
where $ C_l $ is as in \eqref{2.19}. This formula can be explicitly expressed in terms of Stirling numbers of the second kind, as explained in \cite{DW}.

{\color{black}{
\begin{example}\label{example 2}
  \normalfont
For the partitions 
$ ( k ) $ and $ (1^k)  $ in $ \Lambda^{k} $ we get via \eqref{bytheconstruction} that 
$ \dim \Delta_k(k) = \dim \Delta_k(1^k) = 1   $ and so in particular
$  \Delta_k(k) $ and $\Delta_k(1^k)$ are simple $\ParAlg(t)$-modules. 
Suppose that $ n \ge 2k $ such that 
$ \ParAlg(n) $ is semisimple by \cite{PMartin1}. 
Then explicit expressions for the primitive idempotents in $ \ParAlg(n) $
associated
with $ \Delta_k(k) $ and $ \Delta_k(1^k) $ were determined in \cite{BH} and \cite{Campbell}.
In the notation of \cite{Campbell}, these idempotents are the elements 
$ \overline{\textsf{Quasi}_k}$ and $\overline{\textsf{Alt}_k}  $ of $ \ParAlg(n) $. 
\end{example}
}}

We now pass to $\SpheAlg(t)$. 
With the preparations just made we are in position to formulate and prove the promised cellularity of $\SpheAlg(t)$. 
\begin{theorem}\label{cellularity}
  The spherical partition algebra $\SpheAlg(t) $ is cellular on the poset $ \Lambda^k $. 
The cell modules for $\SpheAlg(t) $ are $ \{ e_k  \Delta_k (\lambda) \, | \,  \lambda \in \Lambda^k \}$. 
\end{theorem}
\begin{dem}
  Defining $ \cc = (\s, d , S) \in T(\lambda) $ where $ \lambda =(k) $, 
  $ \s= \T^{\lambda} $ and $ d = S=  \{ \{1 \}, \{2 \}, \ldots, \{k \} \} $, we have 
  $ e_k = \frac{1}{k!}C_{ \cc \cc} $. From this it follows that $ e_k^{\ast} = e_k^{} $ and so we may apply 
Proposition 4.3 of \cite{KoXi}. This proves the Lemma. 
\end{dem}

\medskip
Note that Proposition 4.3 of \cite{KoXi} does not give rise to a basis for $ e_k  \Delta_k (\lambda) $
and in fact our next goal is to construct such a basis. 

\medskip
For this we need several new notational ingredients. 
Suppose first that $ \nu= (\nu_1^{a_1}, \nu_2^{a_2}, \ldots, \nu_p^{a_p} ) \in \Par_i $.
We then define the function 
\begin{equation}
\Psi: \Par_i \rightarrow \Par,  \Psi(\nu) = \ord(a_1, a_2, \ldots, a_p )
\end{equation}
which may be considered as a variation of  
the function $ \Phi$ defined in 
\eqref{weobtainafunction}. 
Define also $ p_i = | \Par_i| $; this is just the classical partition function. 

\medskip
Suppose that $ \s $ is a semistandard $ \lambda $-tableau of type $ \mu $. Following
section 7 in \cite{Murphy}, we now set
\begin{equation}
  x_{\s } = \sum_{ \substack{  w  \in \Si_{\mu}  \\ w \T^{\lambda } \in \Std(\lambda) }}   x_{w \T^{\lambda }} \in S(\lambda). 
\end{equation}  
For example, for $\s = \raisebox{-.45\height}{\includegraphics[scale=0.6]{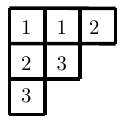}} $
we have
\begin{equation}\label{asin}
  x_{\s} = x_{\raisebox{-.45\height}{\includegraphics[scale=0.6]{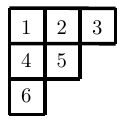}}} +
  x_{\raisebox{-.45\height}{\includegraphics[scale=0.6]{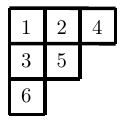}}} +
  x_{\raisebox{-.45\height}{\includegraphics[scale=0.6]{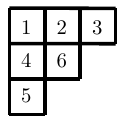}}} +
  x_{\raisebox{-.45\height}{\includegraphics[scale=0.6]{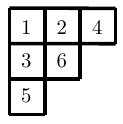}}}.
\end{equation}

Moreover, for any $ \tau \in  \Comp_i $ we define $ d_{\tau} \in \SetPar_i $ as the set partition 
whose blocks are the rows of $ \T^{\tau} $. For example, if $ \tau= (3,2,1,3) $
we get $ d_{\tau} = \{ \{ 1,2,3 \}, \{ 4,5\}, \{6\}, \{7,8,9\}  \}$. 

\medskip
Suppose now that $ \lambda \in \Par_l \subseteq \Lambda^k$ and that $ \nu \in \Par_i $ with 
$ \Psi(\nu) \in \Par_l $ for $ l \le i \le k $. Suppose furthermore that $ \s $ is a semistandard $ \lambda $-tableau of
type $  \Psi(\nu) $ and that $ \mu \in \Par_{k-i}$. Using this information we define an element
$ x_{\nu, \s, \mu} \in e_k \Delta_k(\lambda) $ as follows
\begin{equation}\label{semistandardbasis}
x_{\nu, \s, \mu}  = e_k g \bigl ( ( d_{ \nu \cdot \mu}, d_{\nu}) \otimes  x_{\s} \otimes \dd_{\lambda}  \bigr )
\end{equation}  
where $ \dd_{\lambda} $ is as below \eqref{before dd} and
$ g $ is the isomorphism induced by $ f^{-1} $ for $ f $ as in \eqref{bijection}. 
For example, for $ k = 17 $, $ l = 6$, $ \nu = (3^2,2^2, 1^2) $, $ \mu = (2^2, 1) $ and $ \lambda$ and $ \s $ as
in \eqref{asin}, we have 

\begin{equation} x_{\nu, \s, \mu} = e_{17} \left( \raisebox{-.45\height}{\includegraphics[scale=0.8]{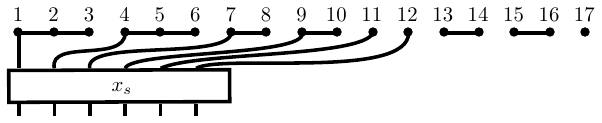}}\right)
\raisebox{-20\height}{.} 
\end{equation}

\medskip
With this notation we can now state and prove the following Theorem. 
\begin{theorem}\phantomsection\label{stateandproveA}
  \begin{description}
  \item[a)] 
    Let $ \lambda \in \Par_l \subseteq \Lambda^k $. Then the set 
    \begin{equation}\label{4.13}
      {\cal B}_{\lambda} = \{ x_{\nu, \s, \mu} \, | \, \nu \in \Par_i  \mbox{ for }   l \le i \le k \mbox{ such that } \Psi(\nu) \in \Par_l, \, 
\s \in \sstd(\lambda, \Psi(\nu)), \, \mu \in \Par_{k-i} \}
    \end{equation}
    is a cellular basis for $ e_k \Delta_k(\lambda) $. 
\item[b)] Suppose that $ \lambda \in \Par_l \subseteq \Lambda^k$. Then we have the
  following dimension formula
  \begin{equation}\label{48}
    \dim e_k \Delta_k(\lambda) =
    \sum_{i=l}^k \sum_{\substack{  \nu \in \Par_i  \\  \Psi(\nu) \in \Par_l }} K_{\lambda,  \Psi(\nu)} p_{{k-i}}
      \end{equation}    
where $ K_{\lambda,  \Psi(\nu)} $ is the Kostka number. 
  \end{description}
\end{theorem}
\begin{dem}
  The right{\color{black}{-}}hand side of \eqref{48} is just the cardinality of $ {\cal B}_{\lambda}$ from ${ \bf a)} $ and so we 
only have to show ${ \bf a)} $. 

\medskip
For this we first recall the set ${\mathcal C}_l$ defined in \eqref{2.19}. 
For $ (c, S) \in  {\mathcal C}_l $ we define
\begin{equation}
  M(c,S) = e_k g \left( (c,S) \otimes {\mathbb C} \Si_l \otimes \dd_{\lambda} \right). 
\end{equation}
We consider $ M(c,S) $ as a right $ {\mathbb C} \Si_l $-module, with action
coming from the right $ \Si_l $-multiplication in the factor $ {\mathbb C} \Si_l $ of $   M(c,S) $. 
For the special element $ e_k g( (c,S) \otimes  1 \otimes \dd_{\lambda} ) \in   M(c,S)$ we let 
$ ( \s, \T )_{\sim_l} $ be the equivalence class of pairs corresponding to $ g( (c,S) \otimes  1 \otimes \dd_{\lambda} ) $ 
under the bijection described in the paragraphs from \eqref {2.21} to \eqref{2.22}. 
The $ \Si_k $-left action on these classes is
faithful and transitive and 
so in the expansion of $ e_k g( (c,S) \otimes  1 \otimes \dd_{\lambda} )  $ there is 
a class represented by a distinguished pair $ ( \s_1, \T^{ (1^l, k-l )}) $ satisfying that the numbers $ \{1, 2, \ldots, k \} $ below the red
line of $ \s_1 $ are all bigger than the numbers above the red line. Moreover, the numbers above the red line of $ \s_1 $ are filled in along rows,
starting with the longest row, followed by the second longest row and so on, and similarly for the numbers below the red line.
In the case of rows of equal lengths, the numbers are filled in
along these rows starting with top one and finishing with the bottom one. Below we give an example of 
$ ( \s, \T^ {  (1^l, k-l) } )_{\sim_l}  $ and its distinguished representative
$ ( \s_1, \T^{  (1^l, k-l) } )  $. 

\begin{equation}\label{distinguished}
     \left( \raisebox{-.5\height}{\includegraphics[scale=0.8]{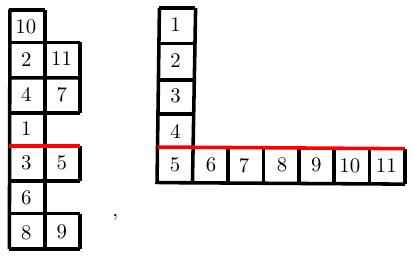}}\right)_{\! \! \! \sim 4} \, \, 
\raisebox{-40\height}{,} 
  \, \,  \, \,  \, \,  \, \,  \, \,  \, \, 
     \left( \raisebox{-.5\height}{\includegraphics[scale=0.8]{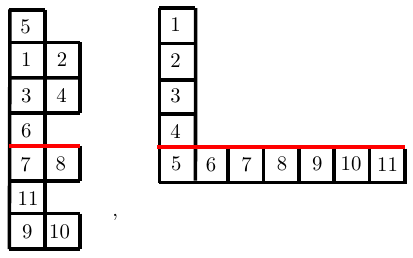}}\right)
  \end{equation}

On the other hand, under the bijection described in the paragraphs from \eqref {2.21} to \eqref{2.22}, the $ \Si_l $-action
on $M(c,S) $
is given by 
row permutations of the top $ l $ rows of the first component of the classes $ ( \s, \T )_{\sim l}  $,
Using this and the description of the distinguished representative for $ ( \s, \T^{  (1^l, k-l) }  )_{\sim_l}  $
just obtained, 
we conclude that $   M(c,S)  $ is isomorphic to the right 
$ {\mathbb C} \Si_l $-permutation module given by $ \Psi(\nu) $, that is $ M( \Psi(\nu) ) \cong
x_{ \Psi(\nu)  \Psi(\nu)} {\mathbb C} \Si_l   
$ where $ \nu =  \ord( \shape({\s_1}_{|1, \ldots, l})) $
for $ {\s_1}_{|1, \ldots, l} $ the restriction of $ \s_1 $ to the first $ l $ rows.

\medskip
We now recall the fact, shown in
\cite{Murphy}, that 
the set $ \{ x_{\s} \, | \, \s \in \sstd(\lambda,  \Psi(\nu)) \} $ is 
a basis for $ x_{ \Psi(\nu)  \Psi(\nu)} S(\lambda) $. Finally taking into account $ \mu =
\ord( \shape({\s_1}_{|l+1, \ldots })) $, where
$ {\s_1}_{|l+1, \ldots } $ is the restriction of $ \s_1 $ to the rows below the red line,
we arrive at the basis given in \eqref{4.13}, which shows that $ {\cal B}_{\lambda}$ indeed is a basis for $e_k \Delta_k(\lambda) $. 

\medskip
Finally, since we already know that the $ e_k \Delta_k(\lambda) $'s are the cell modules for the cellular algebra
$ \SpheAlg(t) $, we get that $ {\cal B}_{\lambda} $ is even a cellular basis for $e_k \Delta_k(\lambda) $. 
{\color{black}{This concludes our proof.}}
\end{dem}

\medskip
\medskip

By cellularity of $\SpheAlg(t) $ we have 
$ \dim  \SpheAlg(t) = \sum_{ \lambda \in \Lambda^k } (\dim e_k  \Delta_k(\lambda) )^2 $, which 
via Theorem \ref{lemma2} and Theorem \ref{stateandproveA} becomes the following 
identity involving $ bp_k $
\begin{equation}\label{415}
  bp_k = \sum_{ \lambda \in \Par_l \subseteq \Lambda^k }\biggl (   \sum_{i=l}^k \sum_{\substack{  \nu \in \Par_i  \\  \Psi(\nu) \in \Par_l }} K_{\lambda,  \Psi(\nu)} p_{{k-i}}
  \biggr)^2. 
\end{equation}
It may be surprising that the identity \eqref{415} can in fact be proved 
with combinatorial tools, as we shall now briefly explain.

\medskip
Fix $\nu \in \Par_i $, $  \mu \in \Par_j $ such that $ \Psi(\nu),  \Psi(\mu) \in \Par_l  $ for
some $ l \in \{0, 1, \ldots, k \} $ 
and consider their contribution to \eqref{415}, that is 
\begin{equation}\label{thissum}
   \sum_{ \lambda \in \Par_l  } K_{\lambda,  \Psi(\mu)} K_{\lambda,  \Psi(\nu)}. 
\end{equation}
The sum in \eqref{thissum} has a
combinatorial interpretation, which is a consequence of the RSK algorithm.

\medskip
Indeed, 
let $ {\mathcal N }_{ \Psi(\mu), \Psi(\nu) } $ be the set of non-negative integer valued matrices 
with row sum $ \Psi(\mu) $ and column sum $  \Psi(\nu)  $.
For example, if $ \mu = ( 2^3, 1^2) $ and $ \nu = (3^2, 2^2, 1)  $ we have 
$ \Psi(\mu) = (3,2) $ and $ \Psi(\nu) = (2,2,1) $ and then ${\mathcal N }_{ \Psi(\mu), \Psi(\nu) } $ 
consists of the matrices
\begin{equation}\label{fivematrices}
\begin{bmatrix*}[r]
  1 & 1 & 1  \\ 1 & 1 & 0  \end{bmatrix*}, \, \, \, 
\begin{bmatrix*}[r]
  0 & 2 & 1  \\ 2 & 0 & 0  \end{bmatrix*}
, \, \, \, 
\begin{bmatrix*}[r]
2 & 0 & 1  \\ 0 & 2 & 0  \end{bmatrix*}
, \, \, \, 
\begin{bmatrix*}[r]
  1 & 2 & 0  \\ 1 & 0 & 1  \end{bmatrix*}
, \, \, \, 
\begin{bmatrix*}[r]
2 & 1 & 0  \\ 0 & 1 & 1  \end{bmatrix*}. 
\end{equation}  
With this notation we have the following formula for \eqref{thissum}, 
see for example Corollary 7.13.2 in \cite{Stanley}
%% {\color{black}{\sout{or 
%% Wikipedia
%% \url{https://en.wikipedia.org/wiki/Robinson-Schensted-Knuth_correspondence}.}}}
\begin{equation}
 \sum_{ \lambda \in \Par_l  } K_{\lambda,  \Psi(\mu)} K_{\lambda,  \Psi(\nu)}  = | {\mathcal N }_{ \Psi(\mu), \Psi(\nu) } | . 
\end{equation}
Now each matrix in ${\mathcal N }_{ \Psi(\mu), \Psi(\nu) } $ corresponds to the propagating part of
an element of $ \BiPar_k $, in the normal form $ GG(b) $ given by Garsia and Gessel, as in \eqref{gg},
with the entries of the matrix giving the number of propagating lines that connect equally sized parts. 
For example, for $ \mu $ and $ \nu $ as above, the five matrices in
${\mathcal N }_{ \Psi(\mu), \Psi(\nu) } $ given by \eqref{fivematrices} correspond to the diagrams
\begin{equation}
\begin{array}{lll}
\raisebox{-.24\height}{\includegraphics[scale=0.8]{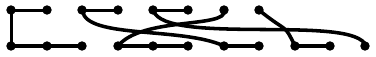}}, & 
\raisebox{-.24\height}{\includegraphics[scale=0.8]{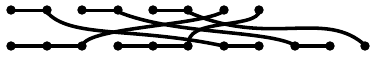}}, &
\raisebox{-.24\height}{\includegraphics[scale=0.8]{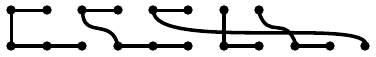}}, \\ && \\
\raisebox{-.24\height}{\includegraphics[scale=0.8]{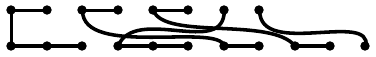}}, & 
\raisebox{-.24\height}{\includegraphics[scale=0.8]{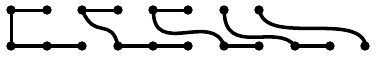}} &
\end{array} 
\end{equation}  
in {\color{black}{the specified}}
order. Using this, and taking into the account the possibilities for the non-propagating part,
we obtain our combinatorial proof of the identity \eqref{415}.

\medskip

We next draw a couple of consequences of Theorem \ref{stateandproveA}. We first
define $ \Lambda_{sph}^{k} \subseteq \Lambda^k $ via
\begin{equation}\label{concretelydefined2}
  \Lambda_{sph}^{k} = \{ \lambda = (\lambda_1, \lambda_2, \ldots, \lambda_l )  \in \Lambda^k \, | \,  \overline{b}(\lambda) 
\le k   \}
\end{equation}
where
\begin{equation} \overline{b}(\lambda) =  \sum_{i=1}^l i \lambda_i.
\end{equation}
This definition 
should be contrasted with the definition of $\ParSph $ in
\eqref{defineLambaSph}. We get
\begin{corollary}\label{sphpar}
With the above notation we have $ e_k \Delta_k(\lambda) \neq 0 $ if and only if $ \lambda \in \Lambda_{sph}^{k}$. 
\end{corollary}
\begin{dem}
  If $ \lambda \in \Lambda_{sph}^{k} $ we consider $ \nu = (l^{\lambda_l}, (l-1)^{\lambda_{l-1}}, \ldots, 1^{\lambda_1} ) $. Then
  $ | \nu |  \le k $ and $ \Psi(\nu )= \lambda $ and so $ K_{\lambda  \Psi(\nu ) } =  K_{\lambda  \lambda } \neq  0 $
  which implies $ e_k \Delta_k(\lambda) \neq 0$, by Theorem \ref{stateandproveA}.
\medskip
  
Suppose now that $ e_k \Delta_k(\lambda) \neq 0$. Then, by Theorem \ref{stateandproveA}, we have $ K_{\lambda \Psi(\nu)} \neq 0 $ for some
partition 
$ \nu  $ with $ | \nu |  \le k $, which implies $ \lambda \unrhd \Psi(\nu) $. Let $ \nu = ( \nu_1^{ a_1}, \nu_2^{ a_2}, \ldots, \nu_l^{ a_l} ) $
where $ \nu_1 > \nu_2 > \cdots > \nu_l $ and suppose that $ \ord(a_1, a_2, \ldots, a_l ) = (b_1, b_2, \ldots, b_l ) $, in other words
$ \Psi(\nu) = (b_1, b_2, \ldots, b_l ) $. Then from $ | \nu |  \le k $ we get 
\begin{equation}\label{which is less than}
\nu_1 a_1 + \nu_2 a_2 + \ldots + \nu_l a_l \le k \Longrightarrow  \nu_1 b_l + \nu_2 b_{l-1} + \ldots + \nu_l b_1 \le k
\Longrightarrow l  b_l  + (l-1) b_{l-1}  + \ldots + 1 b_1 \le k.
\end{equation}
Let now $ \T $ be the semistandard $ \lambda$-tableau of type $ \Psi(\nu) $ that exists
because $ K_{\lambda \Psi(\nu)} \neq 0 $. In $ \T $ the number 1 appears $ b_1 $ times, the number 2 appears $ b_2 $ times etc, 
and so the sum of the numbers appearing in $ \T $ is
$ 1 b_1 + 2 b_2 + \ldots + l b_l  $ which is less than $ k $ by
\eqref{which is less than}.
Let now $ \s $ be the semistandard $ \lambda $-tableau that
is obtained from $ \T $ by replacing each number in $\T $ by the row index of its node. 
The numbers in the $i^{\color{black} \rm th} \,$ row of $ \T $ cannot be strictly less than $ i$, and so also the sum of the numbers in $ \s $ is
smaller than $ k $. On the other hand, $ \s $ is the unique semistandard $\lambda$-tableau of type $ \lambda$ that
has 1 in the nodes of the first row, 2 in the nodes of the second row, etc, and therefore the sum of
numbers in $ \s $ is $ \overline{b}(\lambda) $. This proves the Corollary. 
\end{dem}

%% {\color{black}{
%% \begin{example}\label{example 3}
%%   \normalfont
%% For the partitions 
%% $ ( k ) $ and $ (1^k)  $ considered in Example \ref{example 2},
%% we get via \eqref{concretelydefined2} that $ ( k ) \in \Lambda_{sph}^{k} $
%% but $ ( 1^k ) \notin \Lambda_{sph}^{k} $ if $ k\ge 2 $, or equivalently 
%% $ e_k \overline{\textsf{Quasi}_k}  \neq 0 $ but $e_k \overline{\textsf{Alt}_k} =0 $. This result 
%% can also be obtained directly from the expressions for
%% $ \overline{\textsf{Quasi}_k} $ and $ \overline{ \textsf{Alt}_k } $ found in \cite{BH} and \cite{Campbell}.
%% \end{example}
%% }}

\medskip

It follows from the Corollary that $  \Lambda_{sph}^{k}  $ is a natural parametrizing index set for the representation
theory of $ \SpheAlg(t) $.
Let $ \mathcal A $ be a cellular algebra with cell datum $(\PP,T,\mathcal{C})$ 
as in Definition \ref{cellular} and let $ \{ \Delta(\lambda) \, | \,  \lambda \in \Lambda \} $ be the associated set of cell modules. 
Each $ \Delta(\lambda) $ is endowed with a $ \Bbbk $-valued bilinear form $ \langle \cdot, \cdot \rangle_{\lambda} $
which is important for the representation theory of $ \mathcal A $. To explain $ \langle \cdot, \cdot \rangle_{\lambda} $
one first chooses arbitrarily $ \T_0 \in T(\lambda) $. For basis elements $ C_{\s}, C_{\T} \in \Delta(\lambda) $ 
one considers the expansion of $ C_{\T_0 \T} C_{ \s \T_0 }  $ in the cellular basis for $ \mathcal A$ and then defines
\begin{equation}\label{intheaboveexpansion}
\langle C_{\s}, C_{\T} \rangle_{\lambda}= \coeff_{  C_{\T_0 \T_0 }}(C_{\T_0 \T} C_{ \s \T_0 })
\end{equation}
where $ \coeff_{  C_{\T_0 \T_0 }}(C_{\T_0 \T} C_{ \s \T_0 }) $ is the coefficient of
$ C_{\T_0 \T_0}  $ in the above expansion.

\medskip
Suppose now that $ \Bbbk $ is a field. We define $ \rad(\lambda) =
\{ v \in \Delta(\lambda) \, | \, \langle v , w \rangle_{\lambda} =0 \mbox{ for all } w \in \Delta(\lambda) \} $.
Then $  \rad(\lambda) $ is a submodule of $ \Delta(\lambda) $ and moreover, by the general theory of cellular algebras
developed in \cite{GL}, the quotient module $ L(\lambda) = \Delta(\lambda)/\rad(\lambda) $ is either zero or irreducible,
and the set of nonzero $ L(\lambda)$'s forms a complete set of isomorphism classes for the irreducible $ \mathcal A$-modules.

\medskip
We get the following Theorem. 

\begin{theorem}\label{classiirre}
  Suppose that $ t\notin \{0,1,2, \ldots, 2k-2 \} $. Then $ \SpheAlg(t) $ is semisimple and
  $ \{ e_k \Delta_k(\lambda) \, | \,  \lambda \in \Lambda_{sph}^{k} \} $ is a complete set of representatives for the isomorphism classes 
  of irreducible $ \SpheAlg(t) $-modules. 
\end{theorem}  
\begin{proof}
  We know from Theorem \ref{semisimple} that $ \SpheAlg(t) $ is semisimple. It then follows from Theorem 3.8 of \cite{GL}
  that the nonzero cell modules, that is $ \{ e_k \Delta_k(\lambda) \, | \,  \lambda \in \Lambda_{sph}^{k} \} $,
  are irreducible and pairwise inequivalent.
%{\color{black}{\sout{This shows the Theorem.}}}
\end{proof}

In the following we shall use the language of quasi-hereditary algebras, 
see for example the appendix to \cite{Donkin}.
%{\color{black}{\sout{for a good introduction to it. }}}
In our setting, the following Theorem is useful for us.
\begin{theorem}\label{GLqh}
  $ \mathcal A $ is quasi-hereditary if and only if $ \langle \cdot , \cdot \rangle_{\lambda} \neq 0 $
  for all $ \lambda \in \Lambda$.
\end{theorem}

For $ t \neq 0 $ it is known that $  \ParAlg(t)$ is a quasi-hereditary algebra,  
see \cite{Donkin2} or \cite{KoXi2}. In Theorem \ref{classiirre} we showed that $ \SpheAlg(t) $ is semisimple
and determined its irreducible modules if
$ t \notin \{0, 1,2,\ldots, 2k-2\} $. Combining 
Theorem \ref{stateandproveA} with Theorem \ref{GLqh}, we now obtain the quasi-heredity of 
$ \SpheAlg(t) $ in the remaining cases, except when  $ t=0$. 
\begin{corollary}\label{quasiheriditarySphe}

Suppose that $ t \in \{ 1,2,\ldots, 2k-2\} $. Then   
$\SpheAlg(t) $ is quasi-hereditary on the poset $ \Lambda_{sph}^{k}  $ 
with standard modules $ \{ e_k \Delta(\lambda) \, | \, \lambda \in \Lambda_{sph}^{k} \} $.
\end{corollary}  
\begin{dem}
  Let $ \lambda = (\lambda_1, \lambda_2, \ldots, \lambda_p) \in  \Lambda_{sph}^{k} $ with $ |\lambda |= l $.
  We then construct a special cellular basis element $ x_{\nu, \s, \mu} $
  for $ \Delta_k(\lambda) $ as in \eqref{semistandardbasis}. 
  For $ \nu $ we use $ \nu = (p^{\lambda_p}, (p-1)^{\lambda_{p-1}} ,\ldots, 1^{\lambda_1} ) $ which satisfies $ | \nu |  \le k $ 
  and $ \Psi(\nu ) = \lambda$. For $ \s $ we use the unique semistandard
  $ \lambda $-tableau of type $ \Psi(\nu) $, which has $1 $ in the nodes of the first row, $2$ in the nodes
  of the second row, and so on. Note that $ x_{\s} = x_{\lambda \lambda} $. Finally, for $ \mu $ we use the one-row partition $ \mu = (k-i) $ where
  $|  \nu |= i $. For these choices we set $ C_{\T_0} = x_{\nu, \s, \mu} $ and,
  in view of \eqref{intheaboveexpansion} and Theorem \ref{GLqh}, we must calculate the coefficient
  of $ C_{\T_0 \T_0} $ in the expansion of $ C_{\T_0 \T_0} C_{\T_0 \T_0} $ in terms of the cellular basis for $ \SpheAlg(t) $.
  For example, for $ k = 9$, $ \lambda= (2,2) $, $ \nu= (2^2, 1^2) $ and $ \mu=(3) $ we have diagrammatically 
\begin{equation}\label{intheabove}
  C_{\T_0 \T_0}= \raisebox{-.5\height}{\includegraphics[scale=0.8]{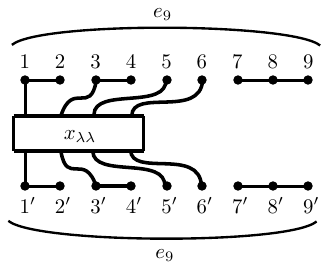}},  \, \, \, \, \, \, \, \, \, \, \, \, \, \, \, \,
  C_{\T_0 \T_0} C_{\T_0 \T_0} = \raisebox{-.5\height}{\includegraphics[scale=0.8]{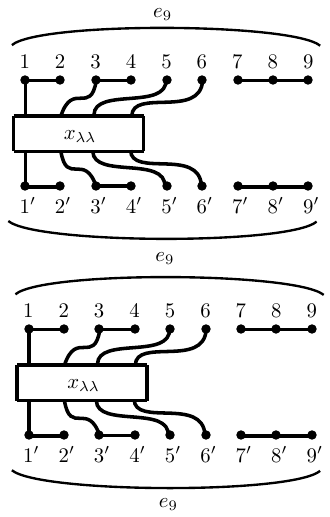}}
\end{equation}    
and must calculate the coefficient of $  C_{\T_0 \T_0}$ in the expansion of $  C_{\T_0 \T_0}  C_{\T_0 \T_0} $. 
For this we first observe that $ x_{\lambda \lambda }^2 = (\prod_{i=1}^p \lambda_i !) x_{\lambda \lambda } $.

\medskip
We next consider the contribution to the coefficient of $ C_{\T_0 \T_0} $ given by $ \sigma \in \Si_k $
from the expansion of the middle $ e_k $ of $ C_{\T_0 \T_0} C_{\T_0 \T_0} $ in terms of the group element basis
of $ \CC \Si_k $. 
We divide the elements $ \sigma \in \Si_k $ in three types, according to their contribution 
to the coefficient of $ C_{\T_0 \T_0}  $ in $ C_{\T_0 \T_0}  C_{\T_0 \T_0}  $. A key point
for what follows is the observation that this division
is exhaustive.

\begin{itemize}
\item[1.] We say that $ \sigma $ is of type 1 if it has the form $ \sigma= \sigma_1 \sigma_2 $ 
  where $ \sigma_1 $ is a permutation of the numbers within blocks of $ d_{\nu \cdot  \mu} $
  and $ \sigma_2 $ is a permutation of the blocks of $  d_{\nu } $ induced by an element from $ \Si_{\lambda} $.
  In the example \eqref{intheabove}, this means that $ \sigma_1 \in \Si_{1,2} \times \Si_{3,4} \times \Si_{7,8,9}
\le \Si_9$ 
and that $ \sigma_2 \in \langle (1,3)(2,4), (5,6) \rangle \le \Si_9$. Each element of type 1 has a contribution of
$ (\prod_{i=1}^p \lambda_i !) \frac{t }{k!} $
to the coefficient of $ C_{\T_0 \T_0}  $ in the product $ C_{\T_0 \T_0}  C_{\T_0 \T_0}  $. Below we give two examples of elements of type 1, the first of the form $ \sigma = \sigma_1 $ and
the second of the form $ \sigma = \sigma_2 $.
\begin{equation}\label{weillustratethisbelow1}
   \raisebox{-.5\height}{\includegraphics[scale=0.8]{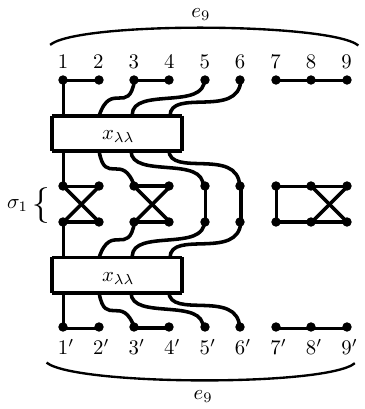}},  \, \, \, \, \, \, \, \, \, \, \, \, \, \, \, \,
   \raisebox{-.5\height}{\includegraphics[scale=0.8]{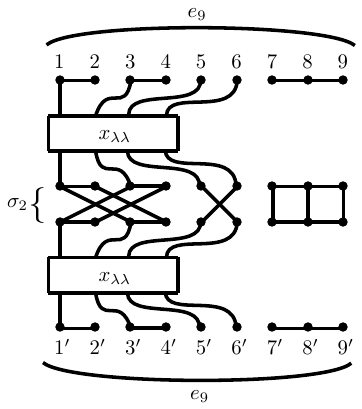}}
\raisebox{-50\height}{.} 
\end{equation}

\item[2.] We say that $ \sigma $ is of type 2 if it has contribution $ (\prod_{i=1}^p \lambda_i !) \frac{1 }{k!} $ to
  the coefficient of $ C_{\T_0 \T_0}  $, in other words, the factor $ t $ appearing in the contribution coming from
  type 1 elements is no longer present. Type 2 elements arise the same way as type 1 elements, except that
  the blocks coming from $ d_{\mu}$ are merged into the other blocks. 
  Below we give an example of an element of type 2.
\begin{equation}\label{weillustratethisbelow2}
   \raisebox{-.5\height}{\includegraphics[scale=0.8]{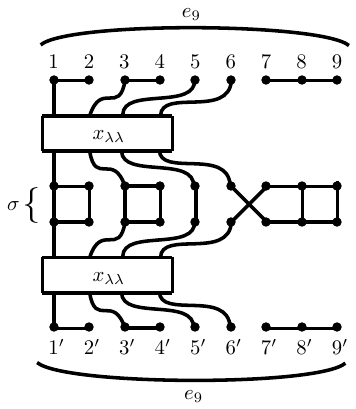}}
\raisebox{-50\height}{.} 
\end{equation}

\item[3.] Finally, we say that $ \sigma $ is of type 3 if it gives rise to a diagram with no contribution to $ C_{\T_0 \T_0}  $
  in the expansion of $  C_{\T_0 \T_0}  C_{\T_0 \T_0}  $, in other words, the diagram in question has strictly fewer than $ l $ propagating blocks. 
  Here is an  example.
\begin{equation}\label{weillustratethisbelow3}
  \raisebox{-.5\height}{\includegraphics[scale=0.8]{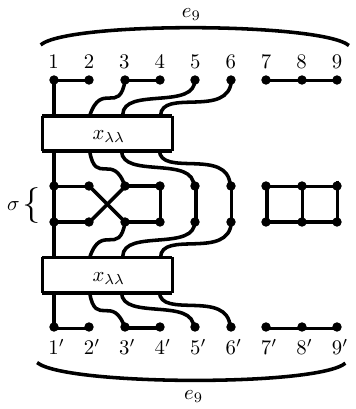}}
\raisebox{-50\height}{.} 
  \end{equation}
\end{itemize}
Let $ A_1, A_2 $ and $ A_3 $ be the cardinalites of type 1, type 2 and type 3 elements, respectively.
The numbers $ A_1, A_2 $ and $A_3 $ can be calculated using combinatorial methods, but we do not need their exact values and shall therefore
not do so. On the other hand, one easily checks that if $ \lambda \neq \emptyset $ then 
$ A_1 > 0 $ whereas $  A_2 > 0 $ if $ \lambda = \emptyset$. 
    
\medskip
Finally, to conclude the proof of the Corollary we now note that the coefficient of $ C_{\T_0 \T_0}  $ in $  C_{\T_0 \T_0}  C_{\T_0 \T_0}  $ is
$  (\prod_{i=1}^p \lambda_i !) \frac{1 }{k!} (A_1 t+ A_2) $ and this is nonzero by the hypothesis on $ t $. 
\end{dem}

\section{The decomposition numbers for $ \SpheAlg(n) $ when $ \SpheAlg(n) $ is non-semisimple. }
\label{decompositionnumbers}
In this section we shall use the results of the previous sections to determine the decomposition numbers for $ \SpheAlg(n) $ when $ \SpheAlg(n) $ is
quasi-hereditary and non-semisimple, that is 
when $ n \in \{1,2,\ldots, 2k-2\} $. 

\medskip
Our arguments depend crucially on \cite{PMartin1} in which the decomposition numbers for $  \ParAlg(n)$ are determined.
The results in \cite{PMartin1} are formulated in terms of the notion of {\it $n$-pairs of partitions}, which we need to explain.
For this, let $ \lambda \in \Par_l $ and let $ u \in  \lambda $ be the $(i,j)^{\color{black} \rm th} \,$node of $ \lambda$. For $ Q \in \Z $ we then define
the {\it $Q $-content of $u $}
as $  c_{\lambda}^Q(u)  = Q +j -i $ and let the {\it $Q $-content diagram of $\lambda $} be 
the diagram obtained from the Young diagram of $ \lambda $ by writing $   c_{\lambda}^Q(u)  $ in each node $ u \in \lambda$. 
For example, for $ \lambda = (5,3,3,2,2)$ the $2$-content diagram is as follows
\begin{equation}\label{nYoung}
   \raisebox{-.5\height}{\includegraphics[scale=0.8]{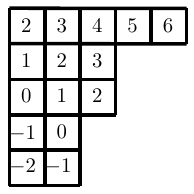}}
\raisebox{-33\height}{.} 
\end{equation}
\begin{definition}
Let $ (\lambda, \mu) $ be a pair of partitions of different orders. We then say that $ (\lambda, \mu) $ is {\it an $n$-pair} if
$ \lambda \subset \mu$ and the Young diagram for $ \mu $ is obtained from the Young diagram for $ \lambda $ by adding nodes
in exactly one row. Furthermore, the rightmost of these nodes should be of $| \lambda | $-content $n$. 
\end{definition}
Below we give two examples of $n$-pairs, in the first we choose $ n=4$ and in the second $ n = 15$.
\begin{equation}\label{npairs}
   \left( \raisebox{-.45\height}{\includegraphics[scale=0.8]{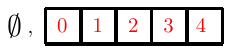}} \right), \, \, \, \, \, \, \, \, \, \, \, \, 
 \left( \raisebox{-.45\height}{\includegraphics[scale=0.8]{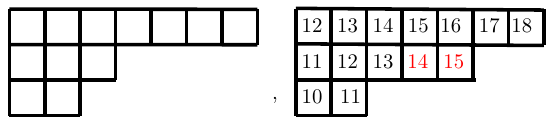}} \right) 
\raisebox{-25\height}{.} 
  \end{equation}

Note that there exists an alcove geometric description of $n$-pairs, see \cite{BdVK}.

\medskip
The following Lemma is immediate from Definition \ref{nYoung}.
\begin{lemma}\label{isimmediate}
  Suppose that $ n \in \Z $ and $ \lambda \in \Par $. Then there exists at most one $ \mu \in \Par $ such that
  $ (\lambda, \mu )$ is an $n$-pair.
\end{lemma}
\begin{proof}
  Let $ \lambda= (\lambda_1^{a_1}, \lambda_2^{a_2}, \ldots, \lambda_p^{a_p}) \in \Par_l$.
  If $ \mu \in \Par $ is obtained from $ \lambda $ by adding nodes to the $i^{\color{black} \rm th} \,$row, then we must have 
  $ i \in  \{1, a_1+1, a_1+a_2+1, \ldots, a_1+a_2+\ldots+a_p +1\} $. Since the $ | \lambda | $-contents are
  constant along the diagonals of $ \lambda $, we conclude from this that the possible values of $ n $ are all distinct,
  which shows the Lemma.
  Below we illustrate on the example $ \lambda = (9^1, 5^3, 3^2) $, where we have indicated with
  red the possible values of $ n$.
\begin{equation}\label{uniquepair}
  \lambda =  \raisebox{-.5\height}{\includegraphics[scale=0.8]{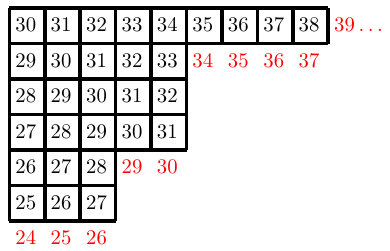}} 
\raisebox{-40\height}{.} 
  \end{equation}

\end{proof}

In \cite{PMartin1} the following important Theorem was proved.
\begin{theorem}\label{mainPaul}
  Let $ n \in \{1,2,\ldots, 2k-2 \} $. For $ \lambda \in \Lambda^k $ let $ L_k(\lambda) = \Delta_k(\lambda)/\rad(\lambda) $
be the irreducible $ \ParAlg(n)$-module associated with $ \lambda$. 
Then the following statements hold. 
  \phantomsection\label{importantThmPaul}
  \begin{description}
  \item[a)] Let $ \lambda, \mu \in \Lambda^k $ with $ \lambda \neq \mu$. Then there is a nonzero homomorphism of $ \ParAlg(n)$-modules
    $ \Delta_k(\mu) \rightarrow \Delta_k(\lambda) $ if and only if $ (\lambda, \mu) $ is an $n$-pair. 
  \item[b)] Let $ \lambda \in \Lambda^k$. If there is no $ \mu \in \Lambda^k $ such that $ (\lambda, \mu ) $ is an $n$-pair then
    $ \Delta_k(\lambda) $ is irreducible. Otherwise, $ \Delta_k(\lambda) $ has decomposition factors $ L_k(\lambda) $ and
    $ L_k(\mu) $ where $ (\lambda, \mu ) $ is the unique $n$-pair with $ \lambda $ in the first factor. 
  \item[c)] Let $ \lambda \in \Lambda^k $ and suppose that $ ( \lambda^1, \lambda^2, \ldots, \lambda^p) $
    is a chain of partitions in $ \Lambda^k$ such that $ \lambda= \lambda^1 $ and such that each $ (\lambda^i, \lambda^{i+1}) $ is an $n$-pair
    for $ i =1,2,\ldots, p-1$. Furthermore, assume that the chain is maximal in the sense that there is no $n$-pair $ (\lambda^p, \mu) $
    with $ \mu \in \Lambda^k$. Then there is a resolution of $ \ParAlg(n)$-modules
    \begin{equation}\label{resolution}
      0 \rightarrow \Delta_k(\lambda^p) \rightarrow \cdots \rightarrow \Delta_k(\lambda^2)\rightarrow \Delta_k(\lambda^1 ) \rightarrow L_k(\lambda  )
\rightarrow 0.
    \end{equation}    
  \end{description}
  
\end{theorem}

Note that \eqref{resolution} gives rise to the formula 
\begin{equation}
\dim L_k(\lambda) = \sum_{i=1}^p (-1)^{i+1} \dim \Delta_k(\lambda^i). 
\end{equation}  
In view of \eqref{bytheconstruction}, this is an explicit formula for $ \dim L_k(\lambda) $. 

\medskip
In order to apply Theorem \ref{mainPaul} we need the following Lemma.
\begin{lemma}\label{weneedthe}
  Suppose that $ \lambda \in \Lambda_{sph}^{k}$. Then $ e_k L_k(\lambda) \neq 0 $. It is an irreducible $ \SpheAlg(n) $-module and
  the set $ \{  e_k L_k(\lambda)  \, | \,  \lambda \in \Lambda_{sph}^{k} \} $ is a complete set of representatives for the
  isomorphism classes of irreducible the $ \SpheAlg(n) $-modules. 
  \end{lemma}
\begin{proof}

In follows from Corollary \ref{quasiheriditarySphe} that $ e_k L_k(\lambda) \neq 0 $ when $ \lambda \in 
\Lambda_{sph}^{k} $. From this the remaining statements of the Lemma follow from the
general cellular algebra theory, see \cite{GL}. 
\end{proof}

\medskip
Combining, we obtain the following Theorem.
\begin{theorem}
\phantomsection\label{mainThm}
\begin{description}
  \item[a)] $ \{  e_k L_k(\lambda)  \, | \,  \lambda \in \Lambda_{sph}^{k} \} $ is a complete set of representatives for the
  isomorphism classes of irreducible the $ \SpheAlg(n) $-modules.
  \item[b)] Let $ \lambda \in \Lambda_{sph}^k$. If there is no $ \mu \in \Lambda_{sph}^k $ such that $ (\lambda, \mu ) $ is an $n$-pair then
    $ e_k\Delta_k(\lambda) $ is an irreducible $ \SpheAlg(n) $-module.
    Otherwise, $ e_k\Delta_k(\lambda) $ has decomposition factors $ e_kL_k(\lambda) $ and
    $ e_kL_k(\mu) $ where $ (\lambda, \mu ) $ is the unique $n$-pair with $ \lambda $ in the first factor. 
  \item[c)] Let $ \lambda \in \Lambda_{sph}^k $ and suppose that $ (\lambda^1, \lambda^2, \ldots, \lambda^p) $
    is a chain of partitions in $ \Lambda_{sph}^k$ such that $ \lambda= \lambda^1 $ and such that each $ (\lambda^i, \lambda^{i+1}) $ is an $n$-pair
    for $ i =1,2,\ldots, p-1$. Furthermore, assume that the chain is maximal in the sense that there is no $n$-pair $ (\lambda^p, \mu) $
    with $ \mu \in \Lambda_{sph}^{k} $. Then there is a resolution of $ \SpheAlg(n)$-modules
    \begin{equation}\label{resolutionA}
      0 \rightarrow e_k\Delta_k(\lambda^p) \rightarrow \cdots \rightarrow e_k\Delta_k(\lambda^2)\rightarrow e_k\Delta_k(\lambda^1 ) \rightarrow e_kL_k(\lambda  )
\rightarrow 0.
    \end{equation}    
  \end{description}
  
\end{theorem}
\begin{proof}
  The statement in ${ \bf a)} $ has already appeared in Lemma \ref{weneedthe}. The statement in ${ \bf c)} $
  follows from ${ \bf c)} $ of Theorem \ref{importantThmPaul} and the fact that left multiplication with $ e_k$ is
  an exact functor. To show the first statement of ${ \bf b)} $, we observe that under the hypothesis on $ \lambda $ the resolution
  \eqref{resolutionA} becomes 
    \begin{equation}\label{resolutionshoert}
      0 \rightarrow e_k\Delta_k(\lambda^1 ) \rightarrow e_kL_k(\lambda  )
\rightarrow 0
    \end{equation}
    which shows that $   e_k\Delta_k(\lambda ) $ is irreducible, as claimed. Finally, the second statement of
    ${ \bf b)} $ follows from the corresponding statement in ${ \bf b)} $ of Theorem \ref{importantThmPaul} and
    exactness of left multiplication with $e_k$.
\end{proof}

As above, we note that the resolution \eqref{resolutionA}, combined with \eqref{48}, gives rise to an explicit formula for
the dimensions of the irreducible $ \SpheAlg(n)$-modules, as follows
\begin{equation}
\dim e_k L_k(\lambda) = \sum_{i=1}^p (-1)^{i+1} \dim e_k \Delta_k(\lambda^i). 
\end{equation}

Let us consider the example $ \lambda = (1) \in \Lambda^3_{sph}$ with $ k=n=3$. 
Then the chain in ${ \bf c)} $ of Theorem \ref{mainThm} has the form $ \{ \lambda^1, \lambda^2 \} $ where
$ \lambda^1 = \lambda $ and $ \lambda^2 = (3) $ and so the resolution in 
\eqref{resolutionA} becomes 
\begin{equation}
0 \rightarrow e_3\Delta_3(\lambda^2 ) \rightarrow e_3\Delta_3(\lambda^1 ) \rightarrow e_3L_3(\lambda  )
\rightarrow 0.
\end{equation}
Using ${ \bf b)} $ of Theorem \ref{stateandproveA} we get $ \dim e_3\Delta_3(\lambda^1 )= 4$ and
$ \dim e_3\Delta_3(\lambda^2 )= 1$ and so we find that $ \dim e_3L_k(\lambda^1 )= 3$. 

\medskip
It is interesting to compare this with $ \dim G_3(\mu) $ where $ \mu =(2,1) \in  \Par_{sph}^{3, 3}$.
Note that $ \overline{\mu} = \lambda $ where $ \overline{\mu} $ 
is defined by $ \overline{ \mu } = ( \mu_2, \ldots, \mu_l)  $ for
$ \mu = (\mu_1, \mu_2, \ldots, \mu_l)$. 
Using
${ \bf c)} $ of Theorem \ref{joining the results} we obtain $ \dim G_3(\mu)  = 3$, that is 
$   \dim G_3(\mu) =  \dim e_3L_k(\lambda ) $.

\medskip 
We think that this equality is no coincidence. To be precise, for $ \lambda \in \ParSph$
we think that it should be true that
\begin{equation}\label{conjecture}
\dim G_k(\lambda ) = \dim e_kL_k(\overline{\lambda} ).
\end{equation}
We note that we have verified \eqref{conjecture} for $ k \le 11 $ using SageMath. We also note that
for $ \ParAlg(n) $ the statement corresponding to \eqref{conjecture} should be true as well but appears not to have been proved
in the literature.

\section{Tilting modules for $ \ParAlg(n) $ and $ \SpheAlg(n)$}\label{tilting}
We already saw that $  \ParAlg(n) $ are quasi-hereditary algebras when 
$ n \neq 0$ and therefore, in particular, they 
are endowed with families of {\it tilting modules}, see the appendix to \cite{Donkin}. 
In this section we take the opportunity to 
describe the structure of these tilting modules, using standard 
arguments from the theory of
quasi-hereditary algebras.
We observe that the same arguments also provide us 
with a description of the tilting modules for $ \SpheAlg(n)$. 

\medskip
We assume $ n \in \{1,2,\ldots, 2k-2\} $
in which case $   \ParAlg(n) $, as we already saw, is non-semisimple quasi-hereditary on the poset $ \Lambda^k $
defined in \eqref{poserpar}. Correspondingly, the category $ \ParAlg(n) $-mod of finite dimensional 
$ \ParAlg(n) $-modules is a {\it highest weight category} where the 
standard modules $ \{ \Delta_k(\lambda) \, | \, \lambda \in 
\Lambda^k \} $ are as described in the paragraphs between \eqref{before dd} and \eqref{6.9} and
the irreducible
modules $ \{ L_k(\lambda) \, | \, \lambda \in 
\Lambda^k \} $ as described in $ { \bf c)} $ of Theorem \ref{mainPaul}.

\medskip
$ \ParAlg(n) $-mod is equipped with a duality $ M \mapsto M^{\ast} $ via 
$ M^{\ast} = \Hom_{\CC}(M, \CC) $ where the $    \ParAlg(n) $-structure on $ M^{\ast} $ is given by 
\begin{equation}
af(m) = f(a^{\ast} m)  \mbox{ for } a \in  \ParAlg(n), f \in M^{\ast},  m \in M
\end{equation}
for $ a \mapsto a^{\ast} $ the anti-automorphism coming from the cellular structure on $    \ParAlg(n) $. 
Note that the $L_k(\lambda) $'s are self dual $ L_k(\lambda) = L_k(\lambda)^{\ast}$ via 
\begin{equation}
L_k(\lambda) \rightarrow L_k(\lambda)^{\ast}, v \mapsto  \langle \cdot , v \rangle_{\lambda} .
\end{equation}
The {\it costandard modules}
$ \{ \nabla_k(\lambda)   \, | \,  \lambda \in \Lambda^k \}$
for $    \ParAlg(n) $ are defined by $ \nabla_k(\lambda) = \Delta_k(\lambda)^{\ast}$.

\medskip
The following definitions and results are part of the general theory of quasi-hereditary algebras. 
Let $ {\cal F}_k(\Delta) $ be the subcategory of $   \ParAlg(n) $-modules
whose objects have $ \Delta$-filtrations, 
in other words, a $   \ParAlg(n) $-module $ M $ belongs to $ {\cal F}_k(\Delta) $ if there is a filtration
of $   \ParAlg(n) $-modules $ 0 \subset M_1 \subset M_2 \subset \ldots \subset M_r = M $ such
that for each $ i=1,2,\ldots, r $ there is a $ \lambda_i \in \Lambda^k$
such that 
$ M_i / M_{i-1} = \Delta_k(\lambda_i ) $.
We define $ {\cal F}_k(\nabla) $ in a similar way, that is $ M\in {\cal F}_k(\nabla) $ if and only if 
$ M^{\ast} \in  {\cal F}_k(\Delta) $. 

\medskip
For $ \lambda \in \Lambda^k $ we let $P_k(\lambda) $ be the projective cover of $ L_k(\lambda) $
in $ \ParAlg(n) $-mod. 
Then $ P_k(\lambda ) \in  {\cal F}_k(\Delta) $ and for any $ \Delta $-filtration
$ 0 \subset M_1 \subset M_2 \subset \ldots \subset M_{r-1} \subset  M_r = P_k(\lambda ) $
with $ M_i/ M_{i-1} = \Delta_k(\lambda_i) $ 
we have $ \lambda_r  =\lambda  $ whereas $ \lambda_j  \rhd \lambda $ for $ j <r$. 
For $ M \in {\cal F}_k(\Delta) $ we define $ (M: \Delta_k(\lambda)) = \dim \Hom_{\ParAlg(n)}(M, \nabla_k(\lambda))$
which is the number of times $ \Delta_k(\lambda) $ occurs as a subfactor in
a $ \Delta$-filtration of $ M$. We then have the Brauer-Humphreys reciprocity formula 
\begin{equation}\label{reciprocity}
(P_k(\lambda): \Delta_k(\mu))  = [ \Delta_k(\mu): L_k(\lambda) ] \, \, \, \mbox{for } \lambda, \mu \in \Lambda^k
\end{equation}  
where $ [ \Delta_k(\mu): L_k(\lambda) ] $ denotes decomposition number multiplicity. 

\medskip
For $ \lambda \in \Lambda^k$ we let $ \ParAlg(n) \mbox{-mod}^{\le \lambda}$ be the subcategory 
of $ \ParAlg(n) $-mod consisting of modules 
with composition factors in $ \{ L_k(\mu) \,|  \, \mu \unlhd \lambda\}$. Then $ \ParAlg(n) \mbox{-mod}^{\le \lambda}$ is a highest weight category with standard modules
$ \{ \Delta_k(\mu) \, | \, \mu \unlhd \lambda \} $ and costandard modules
$ \{ \nabla_k(\mu) \, | \, \mu \unlhd \lambda \} $ and so we deduce from the description
of projective covers 
that $ \Delta_k(\lambda) $ is the projective cover of $ L_k(\lambda) $ in
$\ParAlg(n) \mbox{-mod}^{\le \lambda} $. If $ \mu \lhd \lambda  $ we then get
from $ { \bf b)} $ of Theorem \ref{importantThmPaul} and Proposition A3.3 in \cite{Donkin} that
\begin{equation}\label{84}
\dim   \Ext^1_{\ParAlg(n) \mbox{-mod}}(L_k(\lambda), L_k(\mu) )  =
 \dim \Ext^1_{\ParAlg(n) \mbox{-mod}^{\le \lambda}}(L_k(\lambda), L_k(\mu) )  =
 \left\{ \begin{array}{ll} 1 & \mbox{ if }  (\lambda, \mu ) \mbox{ is an } \,n\mbox{-pair} \\
0 & \mbox{ otherwise } \end{array} \right.
\end{equation}  
and if $ \lambda \lhd \mu  $ we get 
\begin{equation}\label{85}
\dim   \Ext^1_{\ParAlg(n) \mbox{-mod}}(L_k(\lambda), L_k(\mu) )  =
 \dim \Ext^1_{\ParAlg(n) \mbox{-mod}}(L_k(\mu)^{\ast}, L_k(\lambda)^{\ast} )  =
 \left\{ \begin{array}{ll} 1 & \mbox{ if }  (\mu, \lambda ) \mbox{ is an } \,n\mbox{-pair} \\
0 & \mbox{ otherwise } \end{array} \right.
\end{equation}  
since $ L_k(\mu)^{\ast} = L_k(\mu) $ and $ L_k(\lambda)^{\ast} = L_k(\lambda) $.

\medskip
We now fix a chain of partitions ${\cal C}= \{ \lambda^1, \lambda^2, \ldots, \lambda^p \} $ 
in $ \Lambda^k $ such that $(\lambda^i, \lambda^{i+1} ) $ is an $ n$-pair
for $ i=1,2,\ldots, p-1$. Suppose furthermore that the chain is maximal in both directions, in other words 
there is no $ \mu \in \Lambda^k $ such that $ (\mu, \lambda^1) $ is an $n$-pair or such that 
$ (\lambda^p, \mu ) $ is an $n$-pair. By Lemma \ref{isimmediate}, each $ \lambda \in \Lambda^k $ belongs to
a unique such maximal chain $ \cal C$.
Defining
\begin{equation}
  \ParAlg(n) \mbox{-mod}^{\cal C} = \{ M \in \ParAlg(n) \mbox{-mod}  \mid  [M:L_k(\lambda)] \neq 0
  \implies \lambda \in {\cal C}\} 
 \end{equation}
we get from \eqref{84} and \eqref{85} that 
$ \ParAlg(n ) = \oplus_{\cal C} \ParAlg(n) \mbox{-mod}^{\cal C} $ is the block decomposition
of $  \ParAlg(n) \mbox{-mod} $ where $ \cal C$ runs over maximal chains in the above sense.

\medskip
A $ \ParAlg(n)$-module $T$ is called a {\it tilting module} if $ T \in {\cal F}_k(\Delta) \cap {\cal F}_k(\nabla) $. 
For each $ \lambda \in \Lambda^k $ there exists a unique indecomposable tilting module $ T_k(\lambda) $
satisfying $ [ T_k(\lambda): L_k(\lambda) ] = 1 $ and that $ [ T_k(\lambda): L_k(\mu) ] \neq 0
\implies \mu \unlhd \lambda $. 
Each tilting module $ T $ is a direct sum of such $  T_k(\lambda) $'s.

\medskip
Part $   { \bf a)} $ of the following Theorem was obtained already
in \cite{PMartin1}, but still we include it for completeness.

\begin{theorem} With the above notation, we have the following results. 
\phantomsection\label{mainThmLoewy}
\begin{description}
\item[a)] If $ j = 2, 3, \ldots, p-1$ then the Loewy structure for $ P_k(\lambda^j) $ is as follows
  \begin{equation}\raisebox{-.43\height}{\includegraphics[scale=1]{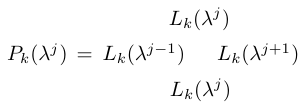}}
\! \!\!\!\!\!\!\!\!\!\!\! \! \!\!\! \!\!\!
\raisebox{-16\height}{.}   
  \end{equation}

\item[b)] If $ j= 1 $ then the Loewy structure for $ P_k(\lambda^1) $ is as follows
\begin{equation}
  \raisebox{-.35\height}{\includegraphics[scale=1]{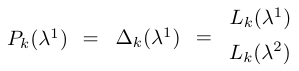}}
\raisebox{-7\height}{.} 
    \end{equation}  
\item[c)]  If $ j= p $ then the Loewy structure for $ P_k(\lambda^p) $ is as follows
\begin{equation}
  \raisebox{-.35\height}{\includegraphics[scale=1]{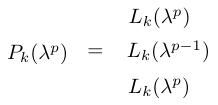}}
 \!\!\!  \!\!
\raisebox{-12\height}{.} 
    \end{equation} 
  \end{description}
  
\end{theorem}
\begin{dem}
  To prove $   { \bf a)} $ we first observe that
$ { \bf b)} $ of Theorem \ref{importantThmPaul} together with \eqref{reciprocity}
imply that  $ (P_k(\lambda^j): \Delta_k(\lambda^i) ) = 1 $ for
$ j = i $ or $ j = i+1$ and otherwise $ (P_k(\lambda^j): \Delta_k(\lambda^i) ) = 0 $. 
Therefore there are two $\Delta$-factors in the $ \Delta$-filtration for $P_k(\lambda^j) $, namely $\Delta_k(\lambda^j) $
and $\Delta_k(\lambda^{j-1}) $. 
On the other hand, defining $ Q_k(\lambda) = {\rm ker}(P_k(\lambda) \rightarrow L_k(\lambda) )$
we get from \eqref{84} and \eqref{85} that $ \dim  \Hom_{\ParAlg(n)}(Q_k(\lambda^j), L_k(\lambda^i)) = 1 $ 
if $ i = j-1 $ or $ i = j+1$ and otherwise $ \dim  \Hom_{\ParAlg(n)}(Q_k(\lambda^j), L_k(\lambda^i)) = 0 $.
Hence the Loewy structure for $ P_k(\lambda^j) $ must be as indicated in $   { \bf a)} $.

To prove $   { \bf b)} $ we once again use
Theorem \ref{importantThmPaul} and \eqref{reciprocity}, but this time we find that $ \Delta_k(\lambda^1) $
is the only $ \Delta $-factor of $P_k(\lambda^1) $, which shows $   { \bf b)} $.

Finally, to show $   { \bf c)}$ we first note that
$   { \bf b)} $ of 
Theorem \ref{importantThmPaul} gives
$ \Delta_k(\lambda^p) = L_k(\lambda^p) $. Since $ P_k(\lambda^p) $ has $ \Delta$-factors
$  \Delta_k(\lambda^p) $ and $  \Delta_k(\lambda^{p-1}) $, as one sees from
Theorem \ref{importantThmPaul} and \eqref{reciprocity}, the structure of $ P_k(\lambda^p) $
must be the one indicated in $   { \bf c)}$. This proves the Theorem. 
\end{dem}

\medskip
We now get the following Theorem, describing the indecomposable tilting modules for $ \ParAlg(n)$.

\begin{theorem}\label{maintilting} The tilting module $ T_k(\lambda^i) $ for $ i=1,2,\ldots, p$ are given by the following.
\phantomsection\label{mainThmLoewyTilting}
\begin{description}
\item[a)] If $ j = 1, 2, \ldots, p-1$ then $  T_k(\lambda^{j}) = P_k(\lambda^{j+1}) $.
\item[b)]  $  T_k(\lambda^p) = \Delta_k(\lambda^p) $.
\end{description}
  \end{theorem}
\begin{dem}
  The modules in $   { \bf a)}$ are described in
  $   { \bf a)}$ and $   { \bf c)}$ of Theorem \ref{mainThmLoewy}. They are self-dual
  and therefore tilting modules. The missing tilting module is $  T_k(\lambda^p) = \Delta_k(\lambda^p) $,
  given in $   { \bf b)}$.
\end{dem}

\medskip
We finally mention that 
there are versions of Theorem \ref{mainThmLoewy} and Theorem \ref{maintilting} 
for $ \SpheAlg(n) $ instead of $ \ParAlg(n) $.  In view of Theorem \ref{mainThm}
the statements and proofs are here 
exactly the same as for Theorem \ref{mainThmLoewy} and Theorem \ref{maintilting}.

 \end{document}